\documentclass[11pt,reqno,final]{amsart}

\usepackage[l2tabu,orthodox]{nag}
\usepackage{latexsym}
\usepackage{mathrsfs}
\usepackage{amssymb}
\usepackage{mathtools}
\DeclareMathAlphabet{\mathbbm}{U}{bbm}{m}{n}
\usepackage{ifdraft}
\usepackage{ifthen}
\newcommand{\Ifempty}[3]{\ifthenelse{\equal{#1}{}}{#2}{#3}}

\usepackage[natbib=true,hyperref=true,backref=true,maxcitenames=2,maxbibnames=10,backend=biber]{biblatex}
\bibliography{higher}

\usepackage{tikz-cd}

\usepackage[breaklinks,bookmarks,final]{hyperref}
\usepackage{showlabels}
\usepackage{verbatim}

\newcommand{\NN}{\mathbb{N}}

\newcommand{\Mod}[1]{\mathrm{#1}}
\newcommand{\mM}{\Mod{M}}

\newcommand{\Cat}[1]{\mathcal{#1}}
\newcommand{\cC}{\Cat{C}}
\newcommand{\Df}[1]{\Cat{D}ef (#1)}

\newcommand{\cop}{\otimes} 
\newcommand{\Th}[1]{\mathcal{#1}}
\newcommand{\tT}{\Th{T}}
\newcommand{\tS}{\Th{S}}
\newcommand{\tC}{\Th{C}}
\newcommand{\tG}{\Th{G}}

\newcommand{\w}{\omega}

\newcommand{\Dset}[1]{\mathbf{#1}}
\newcommand{\dX}{\Dset{X}}
\newcommand{\dY}{\Dset{Y}}
\newcommand{\dG}{\Dset{G}}
\newcommand{\dH}{\Dset{H}}
\newcommand{\dK}{\Dset{K}}
\newcommand{\dL}{\Dset{L}}
\newcommand{\dF}{\Dset{F}}
\newcommand{\dS}{\Dset{S}}
\newcommand{\dQ}{\Dset{Q}}
\newcommand{\dC}{\Dset{C}}
\newcommand{\dP}{\Dset{P}}
\newcommand{\iHom}{\underline{Hom}}
\newcommand{\iIso}{\underline{Iso}}
\newcommand{\iAut}{\underline{Aut}}
\newcommand{\dHom}{\Dset{\iHom}}
\newcommand{\dIso}{\Dset{\iIso}}
\newcommand{\dAut}{\Dset{\iAut}}
\newcommand{\cd}{\Dset{c}}
\newcommand{\dd}{\Dset{d}}

\newcommand{\der}{\partial}
\newcommand{\Cyc}[2][\dG]{{#1}_{#2}^\circ}
\newcommand{\Horn}[3][\dG]{{\Lambda^{#2}_{#3}}({#1})} 
\newcommand{\RHom}{\dHom^R}
\newcommand{\LHom}{\dHom^L}
\newcommand{\Under}[2][\dG]{{{#1}_{#2/}}}

\newcommand{\ti}[1]{\tilde{#1}}
\newcommand{\ra}[1][]{\xrightarrow{#1}}
\newcommand{\iso}{\ra[\sim]}
\newcommand{\tstackrel}[2]{%
  \mathrel{\vbox{\offinterlineskip\ialign{%
    \hfil##\hfil\cr
    #1\cr
    \noalign{\kern-.02em}
    #2\cr
}}}}
\newcommand{\weq}{\tstackrel{\scalebox{2.2}[1]{\(\!\scriptscriptstyle\sim\)}}{\(\dashrightarrow\)}}
\newcommand{\mt}{\mapsto}

\newcommand{\1}{\mathbf{1}}
\newcommand{\x}{\times}

\newcommand{\Def}[1]{\emph{#1}}
\newcommand{\Set}[1]{\{#1\}}
\newcommand{\St}[2]{\Set{#1|#2}}
\newcommand{\pp}[1]{\langle#1\rangle}
\DeclareMathOperator{\dcl}{dcl}
\DeclareMathOperator{\spec}{spec}
\newcommand{\Omit}[1]{{\widehat{#1}}}
\newcommand{\ci}{\Omit{i}}
\newcommand{\cj}{\Omit{j}}
\newcommand{\ck}{\Omit{k}}
\newcommand{\cl}{\Omit{l}}

\DeclareMathOperator{\Hom}{Hom}

\newcommand{\newthm}[2]{\newtheorem{#1}[subsubsection]{#2}}
\newthm{theorem}{Theorem}
\newthm{claim}{Claim}
\newthm{cor}{Corollary}
\newthm{prop}{Proposition}
\newthm{lemma}{Lemma}
\newthm{fact}{Fact}
\theoremstyle{remark}
\newthm{rmk}{Remark}
\newthm{xmpl}{Example}
\newthm{qstn}{Question}
\theoremstyle{definition}
\newthm{defn}{Definition}
\newthm{notation}{Notation}
\newthm{const}{Construction}
\newenvironment{remark}{\begin{rmk}}{\qed\end{rmk}}
\newenvironment{example}{\begin{xmpl}}{\qed\end{xmpl}}

\newenvironment{construction}{\begin{const}}{\qed\end{const}}

\newcommand{\GGII}[1]{\cite[#1]{GGII}}
\newcommand{\HTT}[1]{\cite[#1]{HTT}}
\newcommand{\Kero}[1]{\cite[\href{https://kerodon.net/tag/#1}{#1}]{kerodon}}

\title{Higher internal covers}
\author[M. Kamensky]{Moshe Kamensky}
\address{
  Department of Math \\
  Ben-Gurion University \\
  Be'er-Sheva \\
  Israel
}
\email{\url{mailto:kamensky.bgu@gmail.com}}
\urladdr{\url{https://www.math.bgu.ac.il/~kamenskm}}
\date{\today}
\subjclass{03C40}

\begin{document}

\begin{abstract}
  We define and study a higher-dimensional version of model theoretic 
  internality, and relate it to higher-dimensional definable groupoids in the 
  base theory.
\end{abstract}

\maketitle

\section{Introduction}
The model theoretic notions of \emph{internality} and the binding group came 
up originally in work of Zil'ber on categorical theories (\cite{zilber}), and 
shortly after by Poizat (\cite{poizat}) in the \(\w\)-stable context, where 
it was also noticed that differential Galois theory occurs as a special case.  
The stability hypothesis was completely removed in~\cite[Appendix~B]{comp}, 
where it was shown that the crucial hypothesis is \emph{stable embeddedness} 
of the base theory.

Internality is a condition on a definable set \(\dQ\) in an expansion 
\(\tT^*\) of a theory \(\tT\) to ``almost'' be interpretable in \(\tT\): It 
is interpretable, after adding a set of parameters to \(\tT^*\). In this 
situation, the theory provides a \emph{definable} group \(\dG\) in \(\tT^*\), 
acting definably on \(\dQ\) as its group of automorphisms fixing all elements 
in the reduct \(\tT\). It is important here that the binding group \(\dG\) is 
defined in \(\tT^*\) rather than in \(\tT\): In applications, one often 
understands groups in \(\tT\) better than in \(\tT^*\). The group \(\dG\) 
itself is also internal to \(\tT\), and as a result can be identified with a 
definable group \(\dH\) in \(\tT\), but only non-canonically (and in general, 
only after adding parameters). In the context of differential Galois theory, 
this is related to the fact that the group of points of the (algebraic) 
Galois group of a differential equation does not act on the set of solutions, 
and its identification with the group of automorphisms is not canonical.

The non-canonicity was explained by Hrushovski in~\cite{GGII}, where it is 
shown that the natural object that appears in this context is a definable 
\emph{groupoid} in \(\tT\), with the different groups \(\dH\) occurring as 
the groups of automorphisms of each object. In fact, it is shown there that 
there is a correspondence between groupoids definable in the base theory 
\(\tT\) and internal sorts in expansions of \(\tT\). This correspondence is 
reviewed in~\S\ref{sec:1d}. It is also suggested in~\cite{GGII} that sorts of 
\(\tT^*\) internal to \(\tT\) should be viewed as generalised sorts of 
\(\tT\), obtained as a quotient by the corresponding definable groupoid, just 
like an imaginary sort is obtained from a definable equivalence relation 
(which is a special case). In the current paper we try to follow this 
suggestion, by considering what should be the correct notion of internality, 
after viewing these new sorts as ``legitimate'' definable sorts.

Our approach is motivated by topology. There, a typical example of a groupoid 
arises as the fundamental groupoid \(\pi_1(X)\) of a space \(X\), i.e., the 
groupoid whose objects are the points of \(X\), and whose morphisms are 
homotopy classes of paths. For sufficiently nice \(X\), this groupoid can be 
described in terms of the category of \emph{local systems} (locally constant 
sheaves) on \(X\): each point of \(X\) determines a functor to the category 
of sets, satisfying suitable properties (for example, it commutes with 
products), and each path determines a map between such functors (which 
depends only on the homotopy class since the system is locally constant). We 
propose to view internality as analogous to this picture: definable sets in 
the theory corresponding to a definable groupoid \(\dG\) in \(\tT\) can be 
viewed as local systems (of definable sets) on \(\dG\), and conversely. This 
point of view is explained in~\S\ref{sss:gsets} (the base theory \(\tT\) 
corresponds to a contractible space in this approach, so definable sets in it 
correspond to constant systems).

By definition, the local systems on \(X\) do not tell us anything about the 
homotopy type of \(X\) above homotopical dimension \(1\). To encode higher 
homotopical information, we may try looking at families of spaces rather than 
of sets. A space \(X\) is called \emph{\(n\)-truncated} if \(\pi_k(X,x)\) is 
trivial for all \(k>n\) and base points \(x\in{}X\). Such spaces are 
represented in homotopy theory by what we call in this paper \(n\)-groupoids 
(Definition~\ref{def:ncath} in the definable setting; these are equivalent to 
\(n\)-categories in the sense of~\HTT{\S2.3.4} which are groupoids). In the 
case \(n=1\), these are usual groupoids, and the previous paragraph suggests 
studying them by systems of \(0\)-truncated spaces, i.e., sets. Going one 
dimension higher, one expects to recover \(2\)-groupoids from systems of 
\(1\)-truncated spaces. In the definable context, we decided to identify such 
spaces with internal sorts, we consider ``local systems'' of internal sorts, 
i.e., internal sorts of an expanded theory.

Our main result, Theorem~\ref{thm:main}, shows one direction of this expected 
correspondence: we associate to a \(2\)-groupoid \(\dG\) in the theory 
\(\tT\) a theory \(\tT_\dG\) expanding it, and a collection of internal sorts 
of \(\tT_\dG\), which we view as ``higher local systems''. The statement is 
that the canonical \(2\)-groupoid associated to this datum recovers (up to 
weak equivalence) the original one (part of the other direction is indicated 
briefly, but is mostly left for future work).

We mention that this result is one possible generalisation of the results 
of~\cite{GGII} to higher dimensions.  Other such generalisations include the 
papers~\cite{goodrick},~\cite{rahim} and~\cite{wang}, but they all appear to 
go in different directions. We also mention that in the context of usual 
(rather than definable) homotopy theory, analogous results are well known 
(for example, the main part of Theorem~\ref{thm:main} is really a version of 
the higher dimensional Yoneda lemma), but the methods in the proof of these 
results do not translate easily to the definable setting. In fact, the 
situation here is more similar to the one described in~\HTT{\S6.5}, though 
made simpler by the existence of models (i.e., we have ``enough points'').

\subsection{Structure of the paper}
It is very simple: in Section~\ref{sec:1d}, we review the situation in the 
one dimensional case. This serves both as a motivating analogy and to 
complete some background used later. Most of the material there appears in 
some form in~\cite{GGII} (sometimes implicitly), but we include a few easy 
remarks regarding morphisms and equivalence, interpretation in terms of 
``local systems'', and a different description of the groupoid associated to 
an internal cover (which already appeared slightly differently in~\cite{pv}).

In Section~\ref{sec:2d}, we expose the higher dimensional picture, 
concentrate on dimension \(2\). We first define our higher internal covers, 
then review the theory of (truncated) Kan complexes and \(n\)-categories, 
with a few remarks special to the definable setting, and then prove the main 
result mentioned above (Theorem~\ref{thm:main}).

\subsection{Conventions and terminology}\label{ss:terminology}
For simplicity, we assume our theories \(\tT\) to admit elimination of 
quantifiers. By a \Def{\(\tT\)-structure} we mean a substructure of some 
model of \(\tT\). If \(A\) is such a \(\tT\)-structure, by \(\tT_A\) we mean 
the expansion of \(\tT\) by constants for the elements of \(A\), along with 
the usual axioms describing \(A\). If \(A\) was not mentioned, we mean ``for 
some \(A\)''.

We will also assume \(\tT\) eliminates imaginaries (this could be included in 
the general treatment, but would complicate the exposition). Our usage of 
elimination of imaginaries will often be in the (equivalent) form of the 
existence of internal \(\Hom\)s: For every two definable sets \(\dX\) and 
\(\dY\), there are ind-definable set \(\dHom(\dX,\dY)\) and map 
\(\Dset{ev}:\dX\x\dHom(\dX,\dY)\ra\dY\), identifying the \(A\)-points of 
\(\dHom(\dX,\dY)\), for each \(\tT\)-structure \(A\), with the set of 
\(A\)-definable maps from \(\dX\) to \(\dY\). It follows that the subset 
\(\dIso(\dX,\dY)\) of definable isomorphisms is also ind-definable.

Finally, we assume that each theory is generated by one sort, and finitely 
many relations. Similar to the case in~\cite{GGII}, it can be seen that this 
assumption is not restrictive, since all our constructions commute with 
adding structure.

We recall that an \Def{interpretation} of a theory \(\tT_1\) in another 
theory \(\tT_2\) is a model of \(\tT_1\) in the definable sets of \(\tT_2\): 
it assign definable sets to the elements of the signature of \(\tT_1\), so 
that the axioms in \(\tT_1\) hold (this is often called a \emph{definition} 
in the literature, which is equivalent to an interpretation under our 
assumption of elimination of imaginaries in \(\tT_2\)).  If 
\(i:\tT_1\ra\tT_2\) is such an interpretation, it thus assigns to each 
definable set \(\dX\) of \(\tT_1\) a definable set \(i(\dX)\) of \(\tT_2\).  
Since \(i\) is a model of \(\tT_1\), it assigns definable functions of 
\(\tT_2\) to definable functions of \(\tT_1\), and composition to 
composition, and thus determines a functor from the category \(\Df{\tT_1}\) 
of definable sets of \(\tT_1\) to \(\Df{\tT_2}\).  We will normally identify 
\(i\) with this functor, writing for example \(i(\dX)\) for the 
interpretation of a definable set \(\dX\) of \(\tT_1\). We note that an 
expansion is a particular case of an interpretation.  We remark that not 
every functor from \(\Df{\tT_1}\) to \(\Df{\tT_2}\) arises from an 
interpretation: For example, an interpretation preserves all finite (inverse) 
limits (which always exist in \(\Df{\tT_1}\).) This is the main property of 
such functors that we will use in this paper. A detailed description of 
categories of the form \(\Df{\tT}\) and functors that arise from 
interpretation occurs in~\cite{makkaireyes}, but we will not require it.

Similarly, if \(i,j:\tT_1\ra\tT_2\) are interpretations, a map from \(i\) to 
\(j\) is an elementary map of models (equivalently a homomorphism, by our 
assumption of quantifier elimination), given by definable maps in \(\tT_2\).  
Equivalently, this is a natural transformation of functors, when \(i\) and 
\(j\) are viewed in this way. Such a map is an isomorphism if it has an 
inverse. An interpretation is called a \Def{bi-interpretation} if there is an 
interpretation in the other direction, such that both compositions are 
isomorphic to the identity.

When \(\tT_1\) and \(\tT_2\) are given with fixed interpretations 
\(i_k:\tT\ra\tT_k\) of a theory \(\tT\), we have versions of these notions 
over \(\tT\): an interpretation \(j:\tT_1\ra\tT_2\) is over \(\tT\) if 
\(j\circ{}i_1=i_2\) and given two such interpretations 
\(j_1,j_2:\tT_1\ra\tT_2\), a map \(\alpha:j_1\ra{}j_2\) is over \(\tT\) if 
\(\alpha_{i_1(\dX)}:j_1(i_1(\dX))\ra{}j_2(i_1(\dX))\) is the identity for all 
definable sets \(\dX\) of \(\tT\) (more naturally, we could require a given 
isomorphism from \(j\circ{}i_1\) to \(i_2\), but in practice we can always 
assume it to be the identity, and do that to simplify notation).

\subsection{Acknowledgement and dedication}
I am delighted to thank Tomer Schlank's gang at the Hebrew University (aka 
``Ha'arakya''), and particularly Asaf Horev, Shaul Barkan, Shai Keidar and 
Tomer himself for answering what must have been rather bizarre and elementary 
questions in homotopy theory. They form a very helpful and fun community.

This paper is a late expansion on some vague ideas that I presented in the 
postponed online conference that Honoured Ehud Hrushovski for his 60th 
birthday.  It is a pleasure to thank Udi again for his guidance and dedicate 
the paper (hopefully clearer than my presentation at that talk!) to him.

\section{A review of the classical theory}\label{sec:1d}

\subsection{Stable embeddings and internal covers}
We recall the following classical definition of internal covers:

\begin{defn}
  An expansion \(\tT^*\) of a theory \(\tT\) is an \Def{internal cover} if 
  \(\tT\) is stably embedded in \(\tT^*\), and for some expansion \(\tT^*_A\) 
  of \(\tT^*\) by a set of constants \(A\), each definable set in \(\tT^*_A\) 
  is definably isomorphic to a definable set in \(\tT_{A_0}\), for some set 
  of parameters \(A_0\).
\end{defn}

We recall that \emph{stably embedded} here means that for every definable set 
\(\dX\) in \(\tT\), every subset of \(\dX\) definable in \(\tT^*\) with 
parameters from \(\tT^*\) is definable in \(\tT\), with parameters from 
\(\tT\).

It was noted in~\cite{makkai} that this condition can be reformulated as 
follows: If \(i:\tT\ra\tT^*\) is an expansion and \(\dX,\dY\) are definable 
in \(\tT\), there is a natural (\(\tT^*\)-definable) map 
\(i(\dHom_\tT(\dX,\dY))\ra\dHom_{\tT^*}(i(\dX),i(\dY))\), and \(i\) is stably 
embedded precisely if this map is a bijection (for all \(\dX,\dY\) definable 
in \(\tT\)). We note that in this case, the restriction of this map to the 
subset \(\dIso(\dX,\dY)\) of isomorphisms is also a bijection, and that 
taking into account parameters, no new structure is induced on \(\tT^*\). In 
particular, for any \(\tT\)-structure \(A\), the expansion \({\tT^*}_A\) is 
well defined.

The same definition can be applied to a more general interpretation, so we 
say that an interpretation \(i:\tT_1\ra\tT_2\) is \Def{stable} if for all 
definable sets \(\dX\) and \(\dY\) of \(\tT_1\), the natural 
\(\tT_2\)-definable map
\[i(\dHom_{\tT_1}(\dX,\dY))\ra\dHom_{\tT_2}(i(\dX),i(\dY))\]
is a bijection.  Explicitly, this means that each map definable with 
parameters from \(i(\dX)\) to \(i(\dY)\) (in \(\tT_2\)), ``comes from'' a 
unique map definable with parameters from \(\dX\) to \(\dY\) (in \(\tT_1\)).  
If \(i\) is viewed as a (left-exact) functor, as in~\S\ref{ss:terminology}, 
this is often stated as saying that \(i\) is Cartesian closed.  In these 
terms, the definition of internal covers can be reformulated as follows:

\begin{prop}\label{prp:1intcov}
  An expansion \(\tT^*\) of a theory \(\tT\) is an internal cover if it is 
  stable, and \(\tT^*\) admits a stable interpretation \(p\) in \(\tT_A\) 
  over \(\tT\).
\end{prop}

As mentioned in~\S\ref{ss:terminology}, by ``over \(\tT\)'' we mean that the 
restriction of \(p\) to \(\tT\) coincides with the expansion by constants.

\begin{proof}
  let \(\dQ\) be a definable set in \(\tT^*\), generating it over \(\tT\).  
  Assume first that \(\tT^*\) is an internal cover of \(\tT\), so there is a 
  sort \(\dX\) of \(\tT^*\), an expansion by a constant symbol \(a\in\dX\), and 
  a definable bijection \(g_a:\dQ\ra\dQ_a\), with \(\dQ_a\) definable in 
  \(\tT\). By stable-embeddedness, \(\dQ_a\) is definable by a parameter 
  \(a_0\) in \(\tT\). The assignment \(\dQ\mt\dQ_a\) extends uniquely to an 
  interpretation \(x_{a_0}\) of \(\tT^*\) in \(\tT_{a_0}\), over \(\tT\).  
  Since \(g_a\) determines a definable isomorphism between \(\dQ\) and 
  \(\dQ_a\) (and similarly for any definable set it generates), this 
  interpretation is stable.

  Conversely, assume we have a stable interpretation \(p:\tT^*\ra\tT_{A_0}\) 
  over \(i\). We still denote by \(i\) the extension of \(i\) to the 
  expansion \(\tT_{A_0}\ra\tT^*_{A_0}\) which is the identity on \(A_0\) (it 
  is still stable).  Denoting \(\dC=p(\dQ)\), the set 
  \(\dIso(p(\dQ),p(i(\dC)))=\dIso(\dC,\dC)\) is non-empty, since it contains 
  the identity on \(\dC\). Since \(p\) is stable, the left hand side admits a 
  definable bijection with \(\dIso(\dQ,i(\dC))\), so is non-empty as well.  
  Any point \(a\) of this set shows that \(\tT^*\) is an internal cover.
\end{proof}

\subsection{Definable groupoids}\label{sss:1grpd}
A definable groupoid will be denoted as \(\dG=\pp{\dG_0,\dG_1}\), with 
definable set \(\dG_0\) of objects and a definable set \(\dG_1\) of 
isomorphisms, where the domain and codomain maps are denoted 
\(\dd,\cd:\dG_1\ra\dG_0\), respectively, and composition denoted by 
\(\circ\). For objects \(a,b\in\dG_0\), we write \(\dG(a,b)\) for the 
\(a,b\)-definable set \({\pp{\dd,\cd}}^{-1}(\pp{a,b})\) of morphisms from 
\(a\) to \(b\).  A map \(f:\dG\ra\dH\) of definable groupoids is a definable 
functor: a pair of maps \(f_0:\dG_0\ra\dH_0\) and \(f_1:\dG_1\ra\dH_1\) 
commuting with the domain, codomain and composition maps. We will say that 
\(f\) is a \Def{weak equivalence}, denoted \(f:\dG\weq\dH\), if it induces an 
equivalence of categories on all models (this terminology will be generalised 
in Definition~\ref{def:whe}). Our groupoids will generally \emph{not} be assumed to 
be connected.

In \GGII{\S2}, a definable groupoid is attached to each internal cover. This  
groupoid also admits two descriptions.  The first as a definable groupoid 
\(\dG^*\) in \(\tT^*\): this construction depends on the choice of a 
definable set \(\dX\) in \(\tT^*\), as in the proof of 
Proposition~\ref{prp:1intcov}. Given this choice, the groupoid can be described as 
follows:
\begin{construction}\label{con:groupoid}
  The groupoid associated to the data above is described as follows:
  \begin{description}
    \item[Objects (\(\dG_0\))] Complete types of elements \(a\in\dX\) over 
      \(\tT\), along with an additional object \(*\).  Since \(\tT\) is 
      stably embedded, this set of types is definable in \(\tT\) (it is 
      definable, rather than pro-definable, by our finiteness assumption on 
      the language in~\S\ref{ss:terminology}).
    \item[Morphisms (\(\dG_1\))] The set of isomorphisms from \(*\) to a type 
      \(p\in\dG_0\) is given by the realisations of \(p\).  Given another 
      type \(q\in\dG_0\), a morphism from \(p\) to \(q\) is given by a 
      \(2\)-type \(s\) extending \(p\) and \(q\) (over \(\tT\)). Distinct 
      realisations of \(s\) correspond to distinct ways of writing the 
      morphism \(s\) as a composition of a morphism from \(p\) to \(*\) and a 
      morphism from \(*\) to \(q\).
    \item[Composition] Given a type \(s(x,y)\in\dG_1\) extending 
      \(p(x),q(y)\in\dG_0\), and a type \(t(y,z)\in\dG_1\) extending \(q(y)\) 
      and \(r(z)\in\dG_0\), the internality assumption implies that there is 
      a unique \(3\)-type \(u(x,y,z)\) extending all of them. The restriction 
      \(u\) to \(x,z\) is the composition of \(s,t\).

      The composition of an isomorphism \(a\) from \(*\) to \(p\in\dG_0\) 
      with (the inverse of) another such isomorphism \(b\) to \(q\in\dG_0\) 
      is the type of \(\pp{a,b}\). The other compositions are determined by 
      these conditions.
  \end{description}
\end{construction}

We denote by \(\dG\) the full sub-groupoid of \(\dG^*\) on the same objects 
excluding \(*\).  Then \(\dG\) is defined entirely in \(\tT\).

\subsubsection{}\label{sss:1igrpd}
To give a second description, consider, for each \(\tT\)-structure \(A\), the 
groupoid \(I(A)=I_{\tT^*/\tT}(A)\) whose objects are stable interpretations 
of \(\tT^*\) in \(\tT_A\), that are the identity on \(\tT_A\), and whose 
morphisms are isomorphisms of such interpretations, which are the identity 
when restricted to \(\tT\). Here again we may enlarge \(I\) to obtain 
\(I^*\), by adding an additional object \(*\), which is described explicitly 
as the identity interpretation of \(\tT^*\), and again morphisms are given by 
\(A^*\)-definable isomorphisms of interpretations over \(\tT\) (where \(A^*\) 
is now a \(\tT^*\)-structure). The following statement appeared in a slightly 
different form in~\cite{pv}:

\begin{prop}\label{prp:eqgrpd}
  With notation as in~\ref{sss:1grpd}, for each \(\tT^*\)-structure \(A\), 
  there is fully-faithful embedding \(i_A:\dG^*(A)\ra{}I^*(A)\), preserving 
  the vertex, and commuting with automorphism action on \(A\). If \(A\) is a 
  model, \(i_A\) is an equivalence of categories.
\end{prop}
\begin{proof}
  This is essentially~\GGII{Theorem~3.2}. The functor \(i_A\) was described 
  in the proof of Proposition~\ref{prp:1intcov}: to an object \(p\) of 
  \(\dG(A)\), viewed as a type over \(\tT\) (definable over \(A_0\)), we 
  attach the interpretation \(x_b=x_p\) described there, with \(b\) any 
  realization of \(p\) (as explained there, \(x_b\) depends only on 
  \(p\in{}A_0\) and not on \(b\)). Each such realization determines an 
  isomorphism \(g_b\) from \(x_b\) to \(*\), again as above, which describes 
  the functor on morphisms from \(p\) to \(*\). If \(q\) is another object, 
  with realization \(c\), \(i_A\) assigns 
  \({g_c}^{-1}\circ{}g_b:x_p\ra{}x_q\) to the type \(r\) of the pair 
  \((b,c)\). This depends only on \(r\), since the code for this composition 
  lies in \(\tT\), by stable embeddedness. This code also determines \(r\) 
  completely, so the functor is fully faithful.

  To prove the final statement, let \(i:\tT^*\ra\tT\) be any interpretation 
  over a model \(M_0\). The internality assumption implies that for some 
  \(p\in\dG(M_0)\), the set \(Y\) of isomorphisms between \(*\) and \(p\) is 
  non-empty. Since \(M_0\) is a model, there is a point \(b\) in 
  \(i(Y)(M_0)\). Then \(g_b\) is an isomorphism from \(x_b\) to \(i\).
\end{proof}

To summarise, to each stable embedding of \(\tT\) in \(\tT^*\), we had 
attached a groupoid \(I_{\tT^*/\tT}\) classifying stable interpretations of 
\(\tT^*\) back in \(\tT\). The embedding is an internal cover precisely if 
the groupoid is non-empty, and in this case, the groupoid \(I\) is equivalent 
to a definable one (and to the classical binding groupoid). Conversely, 
starting with a definable groupoid \(\dG\) in \(\tT\), there is an internal 
cover \(\tT^*=\tT_\dG\) and an equivalence \(\dG\ra{}I_{\tT^*/\tT}\):
\begin{construction}\label{con:cover}
  The theory \(\tT_\dG\) expands \(\tT\) by an additional sort \(\dX\), a 
  function symbol \(c:\dX\ra\dG_0\), and a function symbol 
  \(a:\dX\x_{\dG_0}\dG_1\ra\dX\), where 
  \(\dX\x_{\dG_0}\dG_1=\St{\pp{x,g}}{c(x)=d(g)}\), and \(\tT_\dG\) states 
  that \(\dX\) is non-empty, and that the resulting structure is a groupoid 
  \(\dG^*\) extending \(\dG\) by an additional object \(*\), with elements 
  \(x\in\dX\) viewed as morphisms from \(*\) to \(c(x)\), and \(a\) provides 
  the composition of such elements with morphisms of \(\dG\).

  Starting from this \(\tT_\dG\), each element of \(\dX\) exhibits 
  \(\tT_\dG\) as an internal cover of \(\tT\). The type of such an element 
  \(x\in\dX\) over \(\tT\) is given by \(c(x)\), and the realisations of this 
  type are indeed the morphisms from \(*\) to \(c(x)\), so 
  Construction~\ref{con:groupoid} indeed recovers \(\dG\).
\end{construction}

For example, when \(\dG\) is a definable group (groupoid with one object), 
the resulting \(\tT_\dG\) is the theory of \(\dG\)-torsors. We refer 
to~\GGII{\S3} for more details, but note again that our construction is 
slightly different when \(\dG\) is not connected: we always expand just by 
one additional object \(*\), thus obtaining an internal cover (possibly 
incomplete), even in the non-connected case.

Alternatively, Definition~\ref{def:tg2} is a generalisation that also applies 
to this case.

\subsubsection{Definable \(\dG\)-sets}\label{sss:gsets}
If \(\dG=\pp{\dG_0,\dG_1}\) is a definable groupoid in \(\tT\), by a (left) 
\Def{\(\dG\)-set} we mean a definable set \(\dX\), a definable map 
\(\pi:\dX\ra\dG_0\) to the set of objects \(\dG_0\) of \(\dG\), and an 
``action'' map \(a:\dG_1\x_{\dG_0}\dX\ra\dX\), over \(\dG_0\), satisfying the 
usual action axioms (here, \(\dG_1\x{\dG_0}\dX\) is the definable subset of 
\(\dG_1\x\dX\) given by \(d(g)=\pi(x)\), and ``over \(\dG_0\)'' means that 
\(c(g)=\pi(a(g,x))\) for all such pairs).  Thus, a morphism \(g:a\ra{}b\) in 
\(\dG\) determines a bijection \(a_g:\dX_a\ra\dX_b\), where 
\(\dX_t=\pi^{-1}(t)\), and we will sometimes write \(gx\) in places of 
\(a_g(x)\) (a pair \(\pp{\dG,\dX}\) as above is called a \emph{concrete 
groupoid} in~\GGII{\S3}).  We think of \(\dG\)-sets as analogs of local 
systems over \(\dG\). A morphism from a \(\dG\)-set \(\dX\) to another 
\(\dG\)-set \(\dY\) is a definable map from \(\dX\) to \(\dY\) that commutes 
with \(\pi\) and \(a\).

Let \(\dX\) be a \(\dG\)-set. If \(\dH=\pp{\dH_0,\dH_1}\) is another 
groupoid, and \(i:\dG\ra\dH\) is a definable map of groupoids, we set
\[
i_!(\dX)=\St{\pp{h,x}\in\dH_1\x\dX}{i(\pi(x))=\dd(h)}/\sim,
\]
where \(\pp{h,gx}\sim\pp{h\circ{}i(g),x}\) for \(g\in\dG_1\) satisfying 
\(\dd(g)=\pi(x)\) and \(i(\cd(g))=\dd(h)\). This is an \(\dH\)-set, with 
structure map induced by \(\pp{h,x}\mt\cd(h)\) and action induced by 
\(\pp{h',\pp{h,x}}\mt\pp{h'h,x}\). On the other hand, if \(\dY\) is an 
\(\dH\)-set, we set \(i^*(\dY)=\dG_0\x_{\dH_0}\dY\), with the projection to 
\(\dG_0\) as the structure map, and action given by \(\pp{g,y}\mt{}i(g)y\) 
for \(y\in\dY\) with \(\pi(y)=i(\dd(g))\). It is clear that both 
constructions are functorial, and as the notation suggests, \(i_!\) 
is left adjoint to \(i^*\).

With these notions, we have the following description of definable sets in 
\(\tT^*\) as local systems over \(\dG\):
\begin{prop}\label{prp:gsets}
  If \(\tT^*\) is an internal cover of \(\tT\), corresponding to the 
  definable groupoid \(\dG\) in \(\tT\), then the category of definable sets 
  in \(\tT^*\) is equivalent to the category of \(\dG\)-sets in \(\tT\).
  Definable sets from \(\tT\) correspond to themselves, with trivial action.
\end{prop}
\begin{proof}
  To each definable set \(\dX^*\) in \(\tT^*\) we assign the definable set 
  \(\dX=\coprod_{p\in\dG_0}p(\dX^*)\). It follows from the uniformity of 
  \(p\) that \(\dX\) is definable in \(\tT\). By definition, \(\dX\) admits a 
  definable map to \(\dG_0\). The action is given tautologically by the 
  identification of the morphisms in \(\dG_0\) with maps of interpretations.  
  Since each \(p\) is an interpretation, this is functorial in \(\dX^*\).

  In the other direction, let \(\dG^*\) be the canonical extension of \(\dG\) 
  in \(\tT^*\) (we identify \(\dG\) with its image in \(\tT^*\)), let 
  \(i:\dG\ra\dG^*\) be the inclusion, and let \(j\) be the inclusion of the 
  canonical object \(*\) of \(\dG^*\), along with its automorphism group 
  \(\dH\), into \(\dG^*\). A definable \(\dG\)-set \(\dX\) in \(\tT\), viewed 
  again as embedded in \(\tT^*\), corresponds then to \(\dX^*=j^*(i_!(\dX))\) 
  (and the resulting action by \(\dH\) is the natural action by 
  automorphisms.)
\end{proof}

We note that each definable set in \(\tT^*\) comes equipped with an action of 
the binding group \(\dAut(*)\), and with it, the first direction could 
likewise be described as \(\dX=i^*(j_!(\dX^*))\).

\begin{cor}\label{cor:1covistg}
  If \(\dG\) is a groupoid associated to an internal cover \(\tT^*\) of 
  \(\tT\), then \(\tT^*\) is bi-interpretable with \(\tT_\dG\) over \(\tT\).
\end{cor}
\begin{proof}
  The definable groupoid \(\dG^*\) in \(\tT^*\) forms an interpretation of 
  \(\tT_\dG\) over \(\tT\). It is a bi-interpretation since both categories 
  of definable sets are equivalent to the category of \(\dG\)-sets in \(\tT\) 
  (commuting with the above interpretation).
\end{proof}

\subsubsection{Pushouts}
Let \(g:\dK\ra\dG\) and \(h:\dK\ra\dH\) be maps of definable groupoids, and 
assume that \(g\) is fully faithful. We construct another definable groupoid 
\(\dG\cop_\dK\dH\) that can be viewed as the pushout of \(\dG\) and \(\dH\) 
over \(\dK\), as follows: For objects, we let 
\({(\dG\cop_\dK\dH)}_0=\dG_0\coprod\dH_0\). If \(a,b\) are two such objects, 
we define the morphisms as follows:
\begin{enumerate}
  \item If \(a,b\in\dH_0\), then \((\dG\cop_\dK\dH)(a,b)=\dH(a,b)\)
  \item If \(a\in\dG_0\) and \(b\in\dH_0\), morphisms from \(a\) to \(b\) 
    are equivalence classes \(v\cop{}u\) of pairs \(\pp{v,u}\), where 
    \(u\in\dG(a,g(c))\), \(v\in\dH(h(c),b)\) for some \(c\in\dK_0\), and 
    \(\pp{v,g(w)\circ{}u}\) is equivalent to \(\pp{v\circ{}h(w),u}\) for 
    all \(w\in\dK_1\) for which the composition is defined. Morphisms 
    from \(b\) to \(a\) are defined analogously.
  \item If \(a,b\in\dG_0\) are both in the essential image of \(g\), a 
    morphism from \(a\) to \(b\) is similarly defined as an equivalence 
    class \(u'\cop{}v\cop{}u\), with \(u,u'\in\dG_1\) and \(v\in\dH_1\).
  \item If either of \(a,b\in\dG_0\) is not in the essential image of 
    \(g\), then morphisms are the same as in \(\dG\).
\end{enumerate}
The composition \((u'\cop{}v\cop{}u)\circ(u_1'\cop{}v_1\cop{}u_1)\) is 
defined as follows: There are \(a,b\in\dK_0\) such that \(u\circ{}u_1'\) is 
a morphism from \(g(a)\) to \(g(b)\). Since \(g\) is fully faithful, it has 
the form \(g(w)\) for a unique morphism \(w\) from \(a\) to \(b\) in 
\(\dK\). We define the composition to be 
\(u'\cop(v\circ{}h(w)\circ{}v_1)\cop{}u_1\). It is clear that this is 
independent of the choices of representatives. The composition in the other 
cases is defined similarly.

There is an obvious map \(h':\dH\ra\dG\cop_\dK\dH\), and we define 
\(g':\dG\ra\dG\cop_\dK\dH\) by sending each object to itself, each morphism 
between objects not in the essential image of \(g\) to itself as well, and 
for \(a,b\in\dG_0\) in the essential image of \(g\), and \(u\) a morphism 
from \(a\) to \(b\), we set \(g'(u)=(u\circ{}u'^{-1})\cop\1_{h(c)}\cop{}u'\), 
where \(u':a\ra{}g(c)\) is any morphism and \(c\in\dK_0\). We have an 
isomorphism \(\alpha\) from \(h'\circ{}h\) to \(g'\circ{}g\), given on an 
object \(c\in\dK_0\) by \(\1_{g(c)}\cop\1_{h(c)}\). It is routine to check 
that everything is well defined, and also that the following statement holds.

\begin{prop}\label{prp:1pushouts}
  Let \(g:\dK\ra\dG\), \(h:\dK\ra\dH\) and the rest of the notation be as 
  above.
  \begin{enumerate}
    \item Given definable maps of groupoids \(g_1:\dG\ra\dF\) and 
      \(h_1:\dH\ra\dF\), and an isomorphism 
      \(\beta:h_1\circ{}h\ra{}g_1\circ{}g\), there is a unique map of 
      groupoids \(f:\dG\cop_\dK\dH\ra\dF\) that coincides with \(g_1\) and 
      \(h_1\) on the objects, \(f\circ{}g'=g_1\), \(f\circ{}h'=h_1\) and such 
      that \(f\cdot\alpha=\beta\).
    \item \(h':\dH\ra\dG\cop_\dK\dH\) is fully faithful. If \(g\) is a weak 
      equivalence, then so is \(h'\).
  \end{enumerate}
\end{prop}

\begin{remark}
  We could make a similar construction where the set of objects is 
  \(\dG_0\coprod_{\dK_0}\dH_0\) in place of the disjoint union, and with 
  \(\alpha\) the identity. The last proposition provides a map from 
  \(\dG\cop_\dK\dH\) to this variant, which is easily seen to be a weak 
  equivalence. We will use the two constructions interchangeably.
\end{remark}

\begin{remark}
  Without the assumption that one of the maps is fully-faithful, the pushout 
  need not be definable. For example, when all groupoids are groups, this is 
  the usual free product with amalgamation.
\end{remark}

\subsection{Maps of groupoids and of interpretations}\label{sss:1mor}
With stable interpretations over \(\tT\), the assignment 
\(\tT^*\mt{}I_{\tT^*/\tT}\) are contravariantly functorial in \(\tT^*\), and 
fully-faithful: a stable interpretation \(i:\tT_1\ra\tT_2\) over \(\tT\) 
induces a functor \(i^*:I_{\tT_2/\tT}\ra{}I_{\tT_1/\tT}\) by composition.

In the other direction, if \(f:\dG\ra\dH\) is a map of definable groupoids, 
corresponding to internal covers \(\tT_\dG\) and \(\tT_\dH\), \(f\) 
determines a stable interpretation \(i^f:\tT_\dH\ra\tT_\dG\) over \(\tT\), 
that can be described in at least two ways:
\begin{enumerate}
  \item An interpretation of \(\tT_\dH\) over \(\tT\) is determined by its 
    value on the extended groupoid \(\dH^*\) defined in \(\tT_\dH\). We set 
    \(i^f(\dH^*)=\dG^*\cop_\dG\dH\) (with respect to the given map \(f\)).  
    This makes sense since the inclusion of \(\dG\) in \(\dG^*\) is a weak 
    equivalence, and is an interpretation since the embedding of \(\dH\) in 
    \(\dG^*\cop_\dG\dH\) is a weak equivalence that misses precisely one 
    object \(*\), and this completely determines its theory. To see that it 
    is stable, we may first choose a parameter in \(\dG^*\). But then \(i^f\) 
    is identified with one of the standard interpretations into \(\tT\).
  \item Alternatively, we may use Proposition~\ref{prp:gsets} to identify 
    definable sets in \(\tT_\dG\) and in \(\tT_\dH\) with \(\dG\)- and 
    \(\dH\)-sets in \(\tT\). Then \(i^f\) is identified with \(f^*\) (in this 
    approach, it is less direct to see that one gets a stable 
    interpretation).
\end{enumerate}

It is easy to verify that \({(i^f)}^*=f\) (after identifying \(\dG\) with its 
image in \(I_{\tT_G/\tT}\) via~\ref{prp:eqgrpd}, and similarly for \(\dH\)).  
However, not every stable interpretation \(i:\tT_\dH\ra\tT_\dG\) (over 
\(\tT\)) is of the form \(i^f\) for some \(f:\dG\ra\dH\). The other source of 
interpretations comes from the other operation described 
in~\S\ref{sss:gsets}: when \(f\) is a weak equivalence, the composition of 
\(f\) with the inclusion of \(\dH\) in \(\dH^*\) is a weak equivalence, so 
restricting to the image of \(f\) (on the objects), we obtain an 
interpretation \(i_f\) of \(\dG^*\) (hence of \(\tT_\dG\)).

\begin{prop}\label{prp:spans}
  Let \(\dG\) and \(\dH\) be two definable groupoids, with associated covers 
  \(\tT_\dG\) and \(\tT_\dH\). Then every stable interpretation 
  \(i:\tT_\dG\ra\tT_\dH\) over \(\tT\) is obtained as a composition 
  \(i=i^f\circ{}i_g\), for some definable groupoid \(\dK\), definable map 
  \(f:\dH\ra\dK\) and weak equivalence \(g:\dG\weq\dK\).

  In particular, if \(i:\tT_1\ra\tT_2\) is a stable interpretation of 
  internal covers over \(\tT\), we may choose definable groupoids \(\dG_1\) 
  and \(\dG_2\) corresponding to the \(\tT_i\), so that \(i\) is induces by a 
  map \(f:\dG_2\ra\dG_1\) of groupoids (up to bi-interpretation).
\end{prop}

A configuration of the form \(\pp{\dK,f,g}\) as above is called a 
\Def{co-span} from \(\dH\) to \(\dG\).

\begin{proof}
  \(\dH\) embeds in \(I_{\tT_\dH/\tT}\) via Proposition~\ref{prp:eqgrpd}, which maps 
  via \(i^*\) to \(I_{\tT_\dG/\tT}\). We set \(f:\dH\ra\dK\) to be the 
  restriction of \(i^*\) to \(\dH\), where \(\dK\) denotes any definable 
  weakly equivalent subgroupoid of \(I_{\tT_\dG/\tT}\), which also contains 
  \(\dG\).  Then \(g\) is the inclusion of \(\dG\) in \(\dK\).

  The last part follows (using Corollary~\ref{cor:1covistg}) by choosing \(\dG_1\) 
  and \(\dG_2\) arbitrarily, and then replacing \(\dG_1\) by \(\dK\) as 
  above.
\end{proof}

As in the construction of the pushout, we may choose \(\dK\) so that its 
objects are the disjoint union of the objects of \(\dG\) and \(\dH\), and we 
will always assume that this is the case.  In the case when \(i\) is a 
bi-interpretation, we recover the notion of equivalence from~\GGII{\S3}.

\subsubsection{Composition and isomorphisms}\label{sss:composition}
Assume that for groupoids \(\dF\), \(\dG\) and \(\dH\) in \(\tT\), we are 
given interpretations \(i:\tT_\dF\ra\tT_\dG\) and \(j:\tT_\dG\ra\tT_\dH\), 
represented by co-spans \(g_1:\dF\weq\dK_1\), \(f_1:\dG\ra\dK_1\), 
\(g_2:\dG\weq\dK_2\) and \(f_2:\dH\ra\dK_2\) as in~\ref{prp:spans}. Since 
\(g_2\) is a weak equivalence, we may form the pushout 
\(\dK=\dK_1\cop_\dG\dK_2\). By Proposition~\ref{prp:1pushouts}, the map from 
\(\dK_1\) to \(\dK\) is a weak equivalence, and therefore so is the composed 
map \(g\). Hence, \(g\) along with the composed map \(f:\dH\ra\dK\) form a 
co-span that represents a stable interpretation of \(\tT_\dF\) in 
\(\tT_\dH\).  To conform with the decision about the objects of the 
representing groupoid \(\dK\), we remove the intermediate two copies of 
\(\dG_0\), and denote the resulting groupoid by 
\(\dK_2\circ\dK_1=\dK_2\circ_{\dG}\dK_1\) (though it does depend on the 
additional data).  The following is a direct calculation.

\begin{prop}\label{prp:1comp}
  In the above situation, the maps \(g:\dF\weq\dK_2\circ\dK_1\) and 
  \(f:\dH\ra\dK_2\circ\dK_1\) represent the composed interpretation 
  \(j\circ{}i\).
\end{prop}

Finally, we consider isomorphisms of interpretations between (stable) 
interpretations of internal covers over \(\tT\).

\begin{prop}\label{prp:bimor}
  Let \(i,j:\tT_{\dG_1}\ra\tT_{\dG_2}\) be two stable interpretations of 
  internal covers over \(\tT\). Assume \(i\) is represented by a co-span 
  \(i_1:\dG_1\weq\dH_1\) and \(i_2:\dG_2\ra\dH_1\), and \(j\) by  
  \(j_1:\dG_1\weq\dH_2\), \(j_2:\dG_2\ra\dH_2\), where each set of objects of 
  \(\dH_n\) the disjoint union of the objects of \(\dG_1\) and \(\dG_2\) 
  (realized by the object parts of \(i_k\) and \(j_k\)).
  
  Then there is a natural bijection between isomorphisms \(\alpha:i\ra{}j\) 
  (over \(\tT\)) and isomorphisms \(\ti{\alpha}:\dH_1\ra\dH_2\) which are the 
  identity on the images of \(\dG_1,\dG_2\).
\end{prop}

As an example, if \(\dG_1\) and \(\dG_2\) are groups, then each \(\dH_i\) 
corresponds to a \(\dG_1-\dG_2\) bi-torsor, and an isomorphism of the 
corresponding interpretations corresponds to an isomorphism of such 
bi-torsors.

\begin{proof}
  Let \(\dP_l\) be the set of arrows in \(\dH_l\) with domain in \(\dG_1\) 
  and codomain in \(\dG_2\). This is a \(\dG_1\)-set, with structure given by 
  the domain map and composition. The interpretation \(i\) takes \(\dP_1\) to 
  the \(\dG_2\)-set given by viewing the arrows in \(\dP_1\) in the other 
  direction, and likewise with \(j\) and \(\dP_2\). So the map \(\alpha\) 
  maps \(\dP_1\) to \(\dP_2\), compatibly with the composition. This is the 
  same as giving an isomorphism \(\ti{\alpha}\) as in the statement. The rest 
  of the structure is induced by the \(\dP_i\), so \(\alpha\) is determined 
  by \(\ti{\alpha}\). Conversely, each \(\ti{\alpha}\) as in the statement 
  extends to an interpretation.
\end{proof}

We summarise most of the content of this section in the following theorem 
(mostly contained in~\GGII{\S3}):

\begin{theorem}
  Let \(\tT\) be a theory. Each internal cover \(\tT^*\) of \(\tT\) is 
  bi-interpretable over \(\tT\) with an internal cover of the form 
  \(\tT_\dG\). An interpretation of \(\tT_\dH\) in \(\tT_\dG\) corresponds to 
  a co-span from \(\dG\) to \(\dH\), and each such co-span determines an 
  interpretation.  Maps between interpretations correspond to maps between 
  co-spans.

  In particular, covers \(\tT_1\) and \(\tT_2\) are bi-interpretable over 
  \(\tT\) if and only if the corresponding groupoids are equivalent.
\end{theorem}

More succinctly (and slightly more precisely), the bi-category of internal 
covers over \(\tT\) is equivalent to the bi-category of definable groupoids 
in \(\tT\) (with morphisms given by co-spans and morphisms between them given 
by bi-torsors). See also Remark~\ref{rmk:duskin}.

\begin{proof}
  This is a combination of Corollary~\ref{cor:1covistg}, 
  Proposition~\ref{prp:spans} and Proposition~\ref{prp:bimor}.
\end{proof}

The description above exhibits the groupoid associated to an expansion as 
interpretations of \(\tT^*\) in \(\tT\). In~\cite{GGII}, it was suggested 
that definable sets of an internal cover of \(\tT\) can be viewed a 
generalised imaginary sorts of \(\tT\).  With this point of view, it is 
natural to ask for the structure classifying interpretations of such sorts as 
well. However, such generalised sorts have more structure: in addition to the 
sorts themselves and maps between them (interpretations), we also have maps 
between maps. The notion of equivalence should be modified as well: it is no 
longer reasonable to expect a bijection on the level of morphisms. In fact, 
as the \(1\)-dimensional case already shows, it not reasonable to expect even 
a map.

\section{Generalised imaginaries}\label{sec:2d}
We now suggest how internal covers can play the role of definable sets in the 
above description, by going one dimension higher.

\subsection{Higher internal covers}
\begin{defn}
  Let \(\tT\) be a theory, \(\tT_1\) and \(\tT_2\) internal covers of 
  \(\tT\). For every set of parameters \(A\) for \(\tT\), we denote by 
  \(\Hom_\tT(\tT_1,\tT_2)(A)\) the groupoid whose objects are stable 
  interpretations of \(\tT_1\) in \({\tT_2}_A\), over \(\tT_A\), and whose 
  morphisms are isomorphisms of interpretations over \(\tT\).
\end{defn}

Thus, what we denoted above \(I\) is \(\Hom_\tT(\tT^*,\tT)\). Similar to that 
case, \(\Hom_\tT(\tT_1,\tT_2)\) is ind-definable in \(\tT\): if 
\(\tT_i=\tT_{\dG_i}\) for \(\tT\)-definable groupoids \(\dG_1,\dG_2\), each 
interpretation above can be described like in~\S\ref{sss:1mor} as given by 
certain definable maps \(\dG_i\ra\dH\), a definable condition. Similarly, 
isomorphisms between interpretations are given by the \(\tT\)-definable 
families of maps as in Proposition~\ref{prp:bimor}, uniform in the \(\dH_i\) 
(cf.~\S\ref{sss:2gpd} for a more detailed description.)

An interpretation between theories extends to internal sorts: If 
\(i:\tT\ra\tS\) is an interpretation, and \(\ti{\tT}\) is an internal cover 
of \(\tT\), associated to the groupoid \(\dG\) in \(\tT\), we denote by 
\(i(\ti{\tT})\) the internal cover of \(\tS\) associated to \(i(\dG)\).

We now wish to define (slightly) higher analogs of stable embeddings and 
internal covers. One discrepancy with the \(1\)-dimensional case occurs as 
follows: If \(\tT\) is an internal cover of \(\tT_0\), we might be interested 
in only part of the structure on \(\tT\) when considering, for example, the 
Galois group. As long as this partial structure includes the definable sets 
witnessing the internality, this can be done be replacing \(\tT\) with a 
reduct including only those definable sets. In the higher version, definable 
sets are replaced by definable groupoids in \(\tT\) (equivalently, internal 
covers), and again we may wish to restrict to a partial collection. However, 
there is no reason to expect that this partial collection is the full 
collection of definable groupoids in some reduct of \(\tT\). Furthermore, the 
internality condition for \(0\)-definable sets automatically implies it for 
definable sets over parameters. Again, there is no reason to expect a similar 
statement for groupoids. For this reason, our definition depends on the 
auxiliary data \(\Gamma\) consisting of families of definable groupoids, 
which are the groupoids we wish to preserve. More precisely, we have the 
following.

\begin{defn}\label{def:distin}
  Let \(\tT\) be a theory. The data of \Def{distinguished covers} for \(\tT\) 
  consists of the following:
  \begin{enumerate}
    \item An ind-definable family \(\Gamma_0\) of internal covers of \(\tT\) 
      (equivalently, of definable groupoids in \(\tT\)).
    \item An ind-definable family of interpretations over \(\tT\) between any 
      two covers \(\tT_1,\tT_2\in\Gamma_0\), depending definably on 
      \(\tT_1,\tT_2\) and closed under composition (the full definable family 
      is denoted by \(\Gamma_1\)).
    \item For every two interpretations \(f,g:\tT_1\ra\tT_2\) in 
      \(\Gamma_1\), an ind-definable family of isomorphisms from \(f\) to 
      \(g\), closed under composition and restricting to the identity on 
      \(\tT\).  Again we assume that the family of all such isomorphisms is 
      uniformly (ind-) definable in \(f,g\), and denote it by \(\Gamma_2\).
  \end{enumerate}
  If \(\tT_0\) is a reduct of \(\tT\), we will say that 
  \(\Gamma=\pp{\Gamma_0,\Gamma_1,\Gamma_2}\) is over \(\tT_0\) if the 
  parameters for the ind-definable families above range over definable sets 
  in \(\tT_0\).
\end{defn}

We note that in terms of definable groupoids, the closure under composition 
translates to closure under the composition operation 
from~\S\ref{sss:composition}.

If a theory \(\tT\) is given with a collection \(\Gamma\) of distinguished 
covers, we will often omit further explicit reference to \(\Gamma\), and call 
them \Def{admissible covers}.  We modify notions like bi-interpretation etc., 
to be with respect to \(\Gamma\).  In particular, the notation \(\Hom_{\tT}\) 
will refer to admissible covers and admissible maps.

\begin{defn}\label{def:2stable}
  Let \(i:\tT\ra\tS\) be an interpretation, and let \(\Gamma\) be a 
  collection of distinguished covers of \(\tT\).  We say that the 
  interpretation \(i\) is \Def{\(2\)-stable} (with respect to \(\Gamma\)) if 
  for every two internal covers \(\tT_1,\tT_2\) in \(\Gamma\) over each 
  \(\tT\)-structure \(A\), the natural map 
  \(i(\Hom_{\tT_A}(\tT_1,\tT_2))\ra\Hom_{\tS_A}(i(\tT_1),i(\tT_2))\) is an 
  equivalence.

  If \(\Gamma\) is omitted, we take it to be all definable groupoids in 
  \(\tT\), and all definable morphisms among them.
\end{defn}

The expression \(i(\Hom_\tT(\tT_1,\tT_2))\) makes sense, since, as we had 
noted above, \(\Hom_\tT(\tT_1,\tT_2)\) is definable in \(\tT\).

We note:

\begin{prop}
  A stable interpretation \(i:\tT_1\ra\tT_2\) is \(2\)-stable.
\end{prop}
\begin{proof}
  We may replace \(\tT\) by \(\tT_A\), and thus assume that \(A=\emptyset\). 
  Let \(\tT_1,\tT_2\) be internal covers of \(\tT\).  The statement is 
  invariant when replacing each cover with a bi-interpretable one (over 
  \(\tT\)).  Hence, we may assume that \(\tT_1=\tT_{\dG_1}\) and 
  \(\tT_2=\tT_{\dG_2}\), the covers associated to definable connected 
  groupoids \(\dG_1,\dG_2\) in \(\tT\).

  According to Proposition~\ref{prp:spans}, we may choose \(\dG_1\) and 
  \(\dG_2\) so that a stable interpretation of \(i(\tT_1)\) in \(i(\tT_2)\) 
  corresponds to a definable map of groupoids from \(i(\dG_2)\) to 
  \(i(\dG_1)\). Since \(i\) is stable, this map comes from a map in \(\tT\) 
  (and similarly for morphisms).
\end{proof}

The definition of a \(2\)-cover is analogous to that of an internal cover, as 
formulated in~\ref{prp:1intcov}:

\begin{defn}\label{def:2intcov}
  A \Def{\(2\)-internal cover} of a theory \(\tT\) consists of a theory 
  \(\tT^*\), a collection \(\Gamma\) of families of internal covers of 
  \(\tT^*\) over \(\tT\), and a stable embedding \(\tT\ra\tT^*\) such that 
  \(\pp{\tT^*,\Gamma}\) admits a \(2\)-stable interpretation \(p\) in 
  \(\tT_A\), over \(\tT\) (for some set of parameters \(A\)).
\end{defn}

More explicitly, we require that each internal cover in \(\Gamma\) is 
bi-interpretable, over parameters in \(\tT^*\), with one coming from \(\tT\), 
in a manner coherent with interpretations over \(\tT^*\).  Or, via the 
equivalence with groupoids, that for each family of definable groupoids in 
\(\Gamma\) there is a set of parameters \(B\) in \(\tT^*\) such that each 
groupoid in the family is equivalent, over \(B\), to one coming from \(\tT\) 
(again, in a coherent manner).

As in the \(1\)-dimensional case, the typical examples will come from higher 
dimensional groupoids, which we review next.

\subsection{Higher categories and higher groupoids}
We recall a few definitions from homotopy theory and higher category theory, 
adapted to our language and setup. Though our main references are~\cite{HTT} 
and~\cite{kerodon}, the ideas seem to originate in~\cite{joyal} (and in the 
main case of groupoids, much more classically). We will be interested in the 
notions of \(n\)-category and \(n\)-groupoid discussed in~\HTT{\S2.3.4} 
(through most other parts of the paper we are interested in the case \(n=2\), 
but here it is convenient and harmless to work in general.) There, they are 
defined as special cases of quasi-categories and Kan simplicial sets, 
respectively, but for us it will be more convenient to use terminology that 
makes explicit the finite nature of these structures.  The following 
definitions are a variant of the description in~\HTT{\S2.3.4.9}, which gives 
an equivalent condition (in the case of simplicial sets).

For each \(i\in\NN\), we denote by \([i]\) the ordered set 
\(\Set{0,\dots,i}\). For \(k\in{}[i]\), we identify \(k\) with the map 
\([0]\ra{}[i]\) taking \(0\) to \(k\) (writing \(k_i\) if needed), and we let 
\(\ck=\ck_i:[i-1]\ra{}[i]\) be the unique increasing map with \(k\) not in 
the image. We fix a natural number \(n\) (one could also allow \(n=\infty\) 
to obtain the usual definitions of quasi-categories and spaces, but we will 
not use them).

\begin{defn}
  The signature \(\Sigma_n\) of \(n\)-simplicial sets consists of:
  \begin{enumerate}
    \item A sort \(\dG_i\), for \(0\le{}i\le{}n+1\)
    \item For each weakly increasing map \(t:[i]\ra{}[j]\), where 
      \(i,j\le{}n+1\), a function symbol \(d_t:\dG_j\ra\dG_i\).
  \end{enumerate}
\end{defn}

\subsubsection{Notation}
We define the following auxiliary notation. We let \(\dG_{-1}\) be the one 
element set. For each \(0\leq{}m\le{}n+1\), and \(i\le{}m\), the map 
\(d_\ci:\dG_m\ra\dG_{m-1}\) is called the \Def{\(i\)-th face map}. We denote 
by \(\der=\der_m:\dG_m\ra\dG_{m-1}^{m+1}\) the Cartesian product of these 
maps, and by \(\der^\ci:\dG_m\ra\dG_{m-1}^m\) the Cartesian product with 
\(i\) omitted. If \(g\in\dG_m\) and \(t:[k]\ra{}[m]\), we sometimes write 
\(g_t\) in place of \(d_t(g)\) (in particular, for \(t=\ci\) or \(t=i\)).

For \(m\ge{}-1\), the set \(\Cyc{m+1}\) of \Def{\(m+1\)-cycles} is the 
definable subset of \({\dG_m}^{m+2}\) given by the conjunction of the 
equations \(d_\ci(x_j)=d_\ck(x_l)\) for all \(0\le{}j,l\le{}m+1\), 
\(0\le{}i,k\le{}m\) satisfying \(\cj\circ\ci=\cl\circ\ck:[m-1]\ra{}[m+1]\) 
(So no conditions when \(m=0\). Note that \(m+1\)-cycles are potential 
boundaries of \(m+1\)-dimensional elements, but are themselves 
\(m\)-dimensional. This is compatible with the notation in~\cite{HTT}.)

For each \(0\le{}i\le{}m+1\), the set \(\Horn{m+1}{i}\) of \Def{\(i\)-th 
\(m+1\)-horns} is the subset of \({\dG_m}^{m+1}\) defined by the same 
conditions, with \(x_i\) omitted. Hence, the projection 
\(\pi_\ci:\Cyc{m}\ra\dG_{m-1}^m\) omitting the \(i\)-th coordinate takes 
values in \(\Horn{m}{i}\).

We extend the notation by inductively setting \(\dG_m=\Cyc{m}\) for 
\(m>n+1\), and \(d_\ci:\dG_m\ra\dG_{m-1}\) the \(i\)-th projection 
(\(0\le{}i\le{}m\)). Consequently, all the above notation makes sense for 
arbitrary natural number \(m\).

\begin{defn}\label{def:ncath}
  Let \(n\ge{}1\). The theory \(\tC_n\) of \(n\)-categories in this signature 
  says:
  \begin{enumerate}
    \item \(d_{t\circ{}s}=d_s\circ{}d_t\) for \(s,t\) composable, \(d_t\) is 
      the identity whenever \(t\) is. It follows that \(\der_m\) takes values 
      in \(\Cyc{m}\) and \(\der_m^\ci\) in \(\Horn{m}{i}\).
    \item For each \(0<m\le{}n+1\) and each \(0<i<m\), the map 
      \(\der_m^\ci:\dG_m\ra\Horn{m}{i}\) is surjective.
    \item For \(m=n+1,n+2\), \(\der_m^\ci\) is bijective for each \(0<i<m\).
  \end{enumerate}
  The theory \(\tG_n\) of \(n\)-groupoids is the extension of \(\tC_n\) where 
  the conditions above are required to also hold for \(i=0,m\).
\end{defn}

Note that by the third condition, the set \(\dG_{n+1}\) is completely 
determined by the rest of the data. However, it is still convenient to have 
it for the statement of the axiom. It follows from the axioms that the unique 
map \(d_f:\dG_0\ra\dG_m\) is injective, and we will use it to identify 
\(\dG_0\) with its image in each \(\dG_m\), writing \(a\) or \(a^m\) for 
\(d_f(a)\) (this map assigns to each object \(a\) the \(m\)-dimensional 
identity morphisms at \(a\)).

The first condition (when \(n=\infty\)) is the usual definition of a 
simplicial set, the second is the definition of quasi-category (or space, in 
the case of a groupoid, where it is called the Kan condition), and the third 
specifies that the object is an \(n\)-category, rather than a quasi-category.  
By a \Def{definable \(n\)-category} or a \Def{definable \(n\)-groupoid} in 
\(\tT\) we mean an interpretation in \(\tT\) of the respective theory.

The intuition is, roughly speaking, that the horns represent configurations 
of (higher) composable arrows, but the composition (represented by the 
element \(g\)) need not be uniquely determined, except on the highest 
dimension. We refer to the first chapters of~\cite{HTT} for further 
explanations, but explain how the case \(n=1\) of the formalism recovers 
usual categories and groupoids:

\begin{example}\label{exa:1gpd}
  A category can be viewed as a \(1\)-category in the above sense by taking 
  \(\dG_0\) the set of objects, \(\dG_1\) the set of morphisms, and \(\dG_2\) 
  the set of pairs of composable morphisms (as we are forced by the axioms).  
  The maps \(d_\Omit{0},d_\Omit{1}:\dG_1\ra\dG_0\) are the codomain and 
  domain maps, the unique map \(\dG_0\ra\dG_1\) assigns to each object its 
  identity, and the maps \(d_\Omit0,d_\Omit{2},d_\Omit{1}:\dG_2\ra\dG_1\) are 
  the two projections and the composition. The only non-trivial instances of 
  the third conditions are when \(m=2\) and \(i=1\), which asserts that any 
  two composable arrows have a unique composition, and when \(m=3\), which 
  corresponds to associativity of the composition.

  Conversely, each \(1\)-category determines a category by reversing this 
  process (and likewise for groupoids).
\end{example}

As in the \(1\)-dimensional case, the axioms imply that for \(0<i<n+1\), the 
relation \(\Cyc{n+1}\) is the graph of a ``composition'' function 
\(c_i:\Horn{n}{i}\ra\dG_n\), by projecting to the \(i\)-coordinate. For 
\(n\)-groupoids, we also have such maps for \(i=0,n+1\).

\begin{remark}
  If \(\dG\) is an \(n\)-category, and \(m>n\), our extension of the notation 
  determines a canonical way of viewing \(\dG\) as a \(\Sigma_m\) structure, 
  and as such it is an \(m\)-category. Consequently, we will view \(\dG\) as 
  an \(m\)-category for each \(m>n\). If \(\dG\) was an \(n\)-groupoid, it 
  will similarly be an \(m\)-groupoid for \(m>n\).
\end{remark}

\subsubsection{Homotopy sets}
The definition of homotopy sets admit a definable version. Let \(\dG\) be an 
\(n\)-groupoid, and let \(b\in\Cyc{m}\) (\(m\ge{}0\)). We let 
\(\dS(\dG,b)=\der_m^{-1}(b)\) be the set of elements of \(\dG_m\) with 
boundary \(b\). For \(\alpha,\beta\in\dS(\dG,b)\), we write 
\(\alpha\sim\beta\) (or \(\alpha\sim_b\beta\)) if some \(h\in\dG_{m+1}\) 
satisfies \(h_\Omit0=\alpha\), \(h_\Omit1=\beta\) and 
\(h_\ci={d_t(\alpha)}_\ci\), where \(t:[m+1]\ra{}[m]\) is the surjective map 
with \(t(1)=0\) (so \(h\) is a homotopy from \(\alpha\) to \(\beta\), 
relative to the boundary \(b\)).  This is an equivalence relation by the Kan 
condition. Note that when \(m\ge{}n\), this relation coincides with equality.

For \(a\in\dG_0\) and \(k\ge{}0\), we write \(\dS_k(\dG,a)\) for 
\(\dS(\dG,b)\), where \(b\) is the constant boundary with value \(a\) in 
\(\Cyc{k}\).  These are \(a\)-definable sets, whose elements correspond to 
the set of pointed maps from the \(k\)-sphere to \(\dG\) with base point 
\(a\) (note that \(\dK_0(\dG,a)=\dG_0\) does not actually depend on \(a\)).  
The \(k\)-th \Def{homotopy set} of \(\dG\) at \(a\) is the quotient 
\(\pi_k(\dG,a)=\dS_k(\dG,a)/\sim\) (in the case of usual simplicial sets, 
this is one of the equivalent definitions by~\Kero{00W1}).

If \(f:\dG\ra\dH\) is a groupoid map (between definable \(n\)-groupoids in 
the theory \(\tT\)), it commutes with all the structure above, and therefore 
induces definable maps of sets \(\pi_k(f,a):\pi_k(\dG,a)\ra\pi_k(\dH,f(a))\).

\begin{defn}\label{def:whe}
  A definable map \(f:\dG\ra\dH\) of \(n\)-groupoids is a \Def{weak 
  equivalence} if \(\pi_k(f,a):\pi_k(\dG,a)\ra\pi_k(\dH,f(a))\) is a 
  bijection for all \(0\le{}k\le{}n\) and \(a\in\dG_0\).
\end{defn}

\begin{remark}\label{rmk:whe}
  More explicitly, for non-empty \(\dG\), the map \(f:\dG\ra\dH\) is a weak 
  equivalence if and only if the following conditions are satisfied for each 
  \(n\ge{}k\ge{}0\) and each \(a\in\dG_0\):
  \begin{enumerate}
    \item For every \(g_0,g_1\in\dS_k(\dG,a)\), if \(f(g_1)\sim{}f(g_2)\) 
      then \(g_1\sim{}g_2\).
    \item For every \(h\in\dS_k(\dH,f(a))\), there is \(g\in\dS_k(\dG,a)\) 
      with \(f(g)\sim{}h\).
  \end{enumerate}
  Alternatively, \(f\) is a weak equivalence if and only if it induces a 
  surjective map on \(\dS\)-classes, i.e., for each \(\dG\)-cycle \(b\), and 
  each \(v\in\dS(\dH,f(b))\), there is \(u\in\dS(\dG,b)\) with 
  \(f(u)\sim{}v\). (To prove these equivalences, it suffices to show that 
  they hold in each model, where each of these conditions is equivalent to 
  homotopy equivalence,~\Kero{00WV}.)
\end{remark}

\begin{defn}\label{def:1eq}
  The \(n\)-groupoids \(\dG_1\) and \(\dG_2\) are \Def{equivalent} if there 
  are weak equivalences \(f_1:\dG_1\ra\dH\) and \(f_2:\dG_2\ra\dH\) for some 
  \(\dH\).
\end{defn}

\begin{example}
  Let \(\dG,\dH\) be definable groupoids, viewed as definable \(1\)-groupoids 
  as in Example~\ref{exa:1gpd}. A map \(f:\dG\ra\dH\) is a functor. For 
  \(k=0\), the first condition in Remark~\ref{rmk:whe} says that if \(a,b\) 
  are objects of \(\dG\), and there is a morphism between \(f(a)\) and 
  \(f(b)\) in \(\dH\), then there is a morphism from \(a\) to \(b\) in 
  \(\dG\). The second condition says that every object \(h\) of \(\dH\) has a 
  morphism to an object in the image of \(f\). Together, this part implies 
  that \(f\) induces a bijection on isomorphism classes.

  For \(k=1\), the first condition says that if \(g_1,g_2\) are automorphisms 
  of \(a\) such that \(f(g_1)=f(g_2)\), then \(g_1=g_2\), i.e., that \(f\) is 
  faithful. The second condition says that \(f\) is full.

  Hence, \(f\) is a weak equivalence if and only if it is a weak equivalence 
  in the sense of~\S\ref{sss:1grpd}. In particular, our notion of equivalence 
  coincides with~\GGII{\S3}.
\end{example}

As in~\S\ref{sss:composition}, equivalence of \(n\)-groupoids is an 
equivalence relation: If \(\dH\) and \(\dH'\) witness that \(\dG_2\) is 
equivalent to \(\dG_1\) and \(\dG_3\), respectively, the  pushout 
\(\dH\cop_{\dG_2}\dH'\) witnesses the equivalence of \(\dG_1,\dG_3\).

\begin{remark}
  The group operation on \(\pi_k(\dG,a)\) (for \(k>0\)) is also definable, 
  but we will not use this.
\end{remark}

\begin{remark}
  The equivalence of our definitions of homotopy groups and weak equivalence 
  with other formulations that appear, for example, in~\Kero{00V2} does not 
  hold in the definable setting, in general. For example, the analogue of 
  Whitehead's Theorem (\Kero{00WV}) is usually false (as seen already in the 
  one-dimensional setting).
\end{remark}

\subsubsection{Morphism groupoids}
Our next goal is to define the space of morphisms between two objects \(a,b\) 
of an \(n\)-category \(\dG\), and obtain a (weak) version of the Yoneda 
embedding that will make sense in the definable setting.

Let \(\dG\) be an \(n\)-category, and let \(a,b\in\dG_0\) be two objects.  As 
in~\HTT{\S1.2.2}, we define the \(\Sigma_{n-1}\)-structure \(\LHom_\dG(a,b)\) 
by
\begin{equation}
  {\LHom_\dG(a,b)}_k=\St{g\in\dG_{k+1}}{g_0=a,g_\Omit{0}=b^k},
\end{equation}
For \(k\le{}n\). The structure maps are given by \(t\mt{}{d^\dG}_{t^+}\), 
where \(t^+:[u+1]\ra{}[k+1]\) is given by \(t^+(i+1)=t(i)+1\) for 
\(i\in{}[u]\) and \(t^+(0)=0\). It is clear that \(\LHom_\dG(a,b)\) is 
uniformly definable over \(a,b\) when \(\dG\) is definable. It follows 
from~\HTT{\S\S4.2.1.8,2.3.4.18,2.3.4.19} that this structure is equivalent to 
an \(n-1\)-groupoid, but since we are not working up to equivalence, we need 
to prove that it is already an \(n-1\)-groupoid by itself (which we do in 
Proposition~\ref{prp:hom} below).

If we fix a ``generic'' object \(v\in\dG_0\), the assignment 
\(b\mt\LHom(v,b)\) looks like the object part of the Yoneda embedding for 
usual categories. One could hope that this is part of a higher Yoneda 
embedding in our situation as well. However, since there is no composition 
function for morphisms in \(\dG\), such an embedding does not exist as a 
functor (it exists non-canonically for set-theoretic quasi-categories, but not 
definably).  Instead, we have the following situation (explained 
in~\HTT{\S2.1}): There is an \(n\)-category \(\Under{v}\) (\HTT{\S2.3.4.10}), 
defined by \({(\Under{v})}_k=\St{g\in\dG_{k+1}}{g_0=v}\), and a map of 
\(n\)-categories \(\pi:\Under{v}\ra\dG\), given by \(\pi(g)=g_\Omit{0}\). By 
definition, the fibre of this map over \(b\in\dG_0\) is \(\LHom_\dG(v,b)\).  
Moreover, this map is a \emph{left fibration} (\HTT{\S2.1.22}): Given 
\(g\in\Horn[\Under{v}]{k}{i}\), for \(i<k\), any ``filling'' \(h\in\dG_k\) of 
\(\pi(g)\) (so that \(\der^\ci(h)=\pi(g)\)) can be lifted to a filling 
\(\ti{h}\in\Under{v}\) with \(\der^\ci(\ti{h})=g\) and \(\pi(\ti{h})=h\). It 
follows from this that the association \(b\mt\pi^{-1}(b)\) behaves like a 
functor of \(b\), but this is only precisely true in the homotopy category.  

We will show that in the case that \(h\) above is invertible, the lifting 
property above holds for our definable version of equivalence. To do this, we 
show that the map \(\pi\) behaves like a local system: the fibres can be 
continued along (suitable) contractible pieces. The pieces we have in mind 
are defined as follows:
\begin{defn}\label{def:discs}
The simplicial set \(D^l\), for \(l\ge{}0\), is defined by 
\(D^l_k={\Set{0,\dots,l}}^{\Set{0,\dots,k}}\) (all maps, no necessarily 
increasing, from \([k]\) to \([l]\)), with structure maps given by 
composition.
\end{defn}
We will often write elements of \({D^l}_k\) as words of length \(k+1\) in the 
``digits'' \(0,\dots,l\). By the usual Yoneda lemma, maps \(D^l\ra{}D^m\) 
correspond (via composition) to functions 
\(\Set{0,\dots,l}\ra\Set{0,\dots,m}\).  Note that homotopically, all these 
maps are weak equivalences, and in particular the map to the point \(D^0\), 
so that all \(D^l\) are contractible.

We now extend the definition of morphisms as follows: for \(\dG\) a definable 
\(n\)-category, let \(a\in\dG_0\) be an object, and let \(f:D^l\ra\dG\) be a 
map of simplicial sets (perhaps over parameters). We define a 
\(\Sigma_{n-1}\)-structure \(\LHom_\dG(a,f)\) as follows: For each 
\(k\le{}n\),
\begin{equation}
  {\LHom_\dG(a,f)}_k=\St{\pp{g,e}\in\dG_{k+1}\x{D^l}_k}{g_0=a,g_\Omit{0}=f(e)}
\end{equation}
with structure maps given as before by 
\(\pp{g,e}\mt\pp{d_{t^+}(g),e\circ{}t}\) for each weakly increasing function 
\(t:[u]\ra{}[k]\). In other words, \(\LHom_\dG(a,f)\) is the pullback under 
\(f\) of the map \(\pi:\Under{v}\ra\dG\) described above. For \(l=0\) and 
\(f\) mapping the point \(D^0\) to \(b\), we recover the previous definition.  
In general, the projection determines a map \(\LHom_\dG(a,f)\ra{}D^l\) of 
simplicial sets, which can be viewed as the ``restriction'' of \(\pi\) to 
\(D^l\).  If \(h:D^r\ra{}D^l\) is a map of simplicial sets, there is an 
induced map \(\hat{h}:\LHom_\dG(a,f\circ{}h)\ra\LHom_\dG(a,f)\), given by 
\(\hat{h}(\pp{g,e})=\pp{g,h(e)}\).

\begin{prop}\label{prp:hom}
  Let \(\dG\), \(a\in\dG_0\) and \(f:D^l\ra\dG\) be as above.
  \begin{enumerate}
    \item The structure \(\LHom_\dG(a,f)\) is an \(n-1\)-groupoid.
    \item For each map \(h:\Set{0,\dots,r}\ra\Set{0,\dots,l}\) (identified 
      with the corresponding map \(D^r\ra{}D^l\)), the induced map 
      \(\hat{h}:\LHom_\dG(a,f\circ{}h)\ra\LHom_\dG(a,f)\) is a weak 
      equivalence.
  \end{enumerate}
\end{prop}
\begin{proof}
  \begin{enumerate}
    \item Let \(\dH=\LHom_\dG(a,f)\). It is clear that \(\dH\) is a 
      simplicial definable set. To check the Kan condition, we prove a 
      stronger claim, namely, that the projection \(\pi:\dH\ra{}D^l\) is a 
      Kan fibration: given a horn element \(h\in\Horn[\dH]{m}{i}\) and an 
      element \(d\in{}{D^l}_m\) with \(\der_m^\ci(d)=\pi(h)\), there is 
      \(\ti{d}\in\dH_m\) with \(\pi(\ti{d})=d\) and \(\der_m^\ci(\ti{d})=h\).

      Let \(h\in\Horn[\dH]{m}{i}\) be a horn element as above, with 
      \(0\le{}i\le{}m\le{}n+2\).  Such an element is given by a matching 
      sequence of element \(h^j=\pp{g^j,e^j}\), for \(j\in{}[m]\), 
      \(j\ne{}i\), with \(g^j\in\dG_m\) and with 
      \(\pi(h)=\pp{e^0,\dots,e^m}\) an element of \(\Horn[D^l]{m}{i}\). In 
      \(D^l\), each such horn element comes from a \emph{unique} element of 
      \({D^l}_m\). Let \(e\in{}{D^l}_m\) be this element, and let 
      \(g^{-1}=f(e)\in\dG_m\).

      We claim that \(\ti{g}=\pp{g^{-1},g^0,\dots,g^m}\in\dG_m^{m+1}\) is in 
      \(\Horn{m+1}{i+1}\). To show that, we need to show that if 
      \(\Omit{b}\circ\Omit{a}=\Omit{d}\circ\Omit{c}:[m-1]\ra{}[m+1]\) for 
      some \(b,d\in{}[m+1]\), \(b,d\ne{}i+1\) and \(a,c\in{}[m]\), then 
      \({g^{b-1}}_\Omit{a}={g^{d-1}}_\Omit{c}\).

      Assume first that \(a,b,c,d\ge{}1\). The assumption on \(h\) implies 
      that
      \begin{equation}\label{eq:compath}
        {h^{b-1}}_{\Omit{a-1}_{m-1}}={h^{d-1}}_{\Omit{c-1}_{m-1}}
      \end{equation}
      whenever \(a,b,c,d\) satisfy
      \begin{equation}\label{eq:compatm}
        \Omit{b-1}_m\circ\Omit{a-1}_{m-1}=\Omit{d-1}_m\circ\Omit{c-1}_{m-1}
      \end{equation}
      Equation~\eqref{eq:compath} implies that 
      \({g^{b-1}}_{{\Omit{a-1}_{m-1}}^+}={g^{d-1}}_{{\Omit{c-1}_{m-1}}^+}\) 
      under this condition. But \({\Omit{j}_{m-1}}^+=\Omit{j+1}_m\) for all 
      \(j\in{}[m-1]\), so we find that 
      \({g^{b-1}}_{\Omit{a}_m}={g^{d-1}}_{\Omit{c}_m}\) whenever 
      Equation~\eqref{eq:compatm} holds. But Equation~\eqref{eq:compatm} is equivalent 
      to
      \begin{equation}
        \Omit{b}_{m+1}\circ\Omit{a}_m=\Omit{d}_{m+1}\circ\Omit{c}_m
      \end{equation}
      so we obtain the required condition when \(a,b,c,d\ge{}1\).

      If \(b=0\) or \(d=0\), the corresponding element of \(\dG_m\) is 
      \(g^{-1}\). In this case, the condition follows from the definition of 
      \(\LHom_\dG(a,f)\): for example, if \(b=0\) we must have \(c=0\) and 
      \(d=a+1\), so we need to show that 
      \({g^a}_\Omit{0}={f(e)}_\Omit{a}=f(e_\Omit{a})=f(e^a)\), and we are 
      done.  If \(a=0\) or \(c=0\), the condition forces \(b=0\) or \(d=0\), 
      so we are back to the same case.

      This concludes the proof that \(\ti{g}\in\Horn{m+1}{i+1}\). If \(i<m\), 
      the Kan condition on \(\dG\) implies that that we may find 
      \(g\in\dG_{m+1}\) restricting to \(\ti{g}\). In follows that \(g_0=a\) 
      and \(g_\Omit0=f(e)\), so that \(\pp{g,e}\) solves the lifting problem.  
      It follows from~\HTT{\S1.2.5.1} that the case \(i<m\) is sufficient.

      When \(m=n\) or \(m=n+1\), the injectivity follows similarly from 
      injectivity for \(\dG\) (and for \(D^l\)).
    \item We use Remark~\ref{rmk:whe}. An element in 
      \({\LHom_\dG(a,f\circ{}h)}_0\) is given by \(g\in\dG_1\) with \(g_0=a\) 
      and \(g_1=f(h(e))\), where \(e\in{}[u]\). Assume that 
      \(\pp{s,c},\pp{t,d}\in{\LHom_\dG(a,f\circ{}h)}_k\) satisfy 
      \(s_\Omit0=t_\Omit0=f(h(c))=f(h(d))=h(e)\) and 
      \(s_\Omit{i}=t_\Omit{i}=g\) for \(k\ge{}i>0\), so that they are 
      elements of \(\dS_k(\LHom_\dG(a,f\circ{}h))\). Assume also that we are 
      given some \(w\in\dG_{k+2}\) satisfying \(w_\Omit1=s\), \(w_\Omit2=t\) 
      and \(w_\Omit{i}=g\) for \(i>2\), and some \(v\in{}{D^u}_{k+1}\) with 
      \(v_\Omit{0}=c\), \(v_\Omit1=d\) and \(v_\Omit{i}=e\) for \(i>1\), and 
      with \(f(h(v))=w_\Omit0\) (this is a homotopy from \(\pp{s,c}\) to 
      \(\pp{t,d}\)). Then \(\pp{w,h(v)}\) is a homotopy from \(\pp{s,h(c)}\) 
      to \(\pp{t,h(d)}\). The argument for \(k=0\) is similar (using that 
      \(D^u\) is connected).

      For the second condition of Remark~\ref{rmk:whe}, let \(g\in\dG_1\) be 
      such that \(g_0=a,g_1=f(e)\) for some \(e\in{}[l]\), so that 
      \(b=\pp{g,e}\) represents a basepoint of \(\LHom_\dG(a,f)\), and let 
      \(\pp{s,c}\in\dS_k(\LHom_\dG(a,f),b)\). Then \(c\in\dS_k(D^l,e)\) is 
      the constant function \(e\). Let \(e'\in{}[u]\), and let 
      \(\gamma\in{}D^l\) be some path from \(h(e')\) to \(e\). By the Kan 
      condition above, there is an element \(s'\) of \(\LHom_\dG(a,f)\) above 
      \(\gamma\), restricting to \(s\). This \(s'\) serves as a homotopy from 
      \(s\) to an element over \(h(e')\), which is thus in the image of 
      \(\hat{h}\).\qedhere
  \end{enumerate}
\end{proof}

\begin{cor}\label{cor:hom}
  Let \(\dG\) be an \(n\)-category, \(v,a,b\in\dG_0\) objects. Each 
  isomorphism \(t\in\dG_1\) from \(a\) to \(b\) determines an equivalence 
  \(e_t:\LHom_\dG(v,a)\ra\LHom_\dG(v,b)\). If \(s\in\dG_1\) is another 
  isomorphism from \(a\) to \(b\), each isomorphism \(m:t\ra{}s\) determines 
  an isomorphism \(e_m:e_t\ra{}e_s\).
\end{cor}
\begin{proof}
  Apply Proposition~\ref{prp:hom} to maps from \(D^1\) and from \(D^2\) 
  determined by \(t\), \(s\) and \(e\).
\end{proof}

We describe the equivalence explicitly in the \(2\)-dimensional case, which 
will be most relevant for us:

\begin{example}\label{exa:hom2gpd}
  Let \(\dG\) be a \(2\)-groupoid, \(v,a,b\in\dG_0\) and \(f\in\dG_1\) with 
  \(f_0=a\) and \(f_1=b\). Since \(\dG\) is a groupoid, there is \(h\in\dG_2\) 
  (not necessarily unique) with \(h_\Omit0=f\) and \(h_\Omit1=b\). We denote 
  \(f^{-1}=h_\Omit2\). By the \(2\)-groupoid axioms, \(h\) has a uniquely 
  determined inverse \(h^{-1}\). Let \(\gamma:D^1\ra\dG\) be the unique map 
  with \(\gamma(101)=h\), so that \(\gamma(01)=g\) and \(\gamma(10)=g^{-1}\).  
  Then \(\dH=\LHom_\dG(v,\gamma)\) can be described as follows:
  \begin{enumerate}
    \item \(\dH_0={\LHom_\dG(v,a)}_0\coprod{}{\LHom_\dG(v,b)}_0\) (this is 
      just the union if \(a\ne{}b\), but if \(a=b\) we take disjoint copies).
    \item
      Let \(\dX=\St{g\in\dG_2}{g_0=v,g_\Omit0=f}\), and let 
      \(\dX^{-1}=\St{g\in\dG_2}{g_0=v,g_\Omit0=f^{-1}}\) (again taking 
      disjoint copies if \(f=f^{-1}\)). Then
      \[
      \dH_1={\LHom_\dG(v,a)}_1\coprod{}{\LHom_\dG(v,b)}_1\cup\dX\cup\dX^{-1}
      \]
      with \(d_\Omit0^\dX=d_\Omit1^\dG\), \(d_\Omit1^\dX=d_\Omit2^\dG\) and 
      vice versa for \(\dX^{-1}\) (and the structure coming from \(\LHom\) on 
      the other parts).
    \item Composition is defined again as in \(\LHom\) on the corresponding 
      parts. The composition of \(h\circ{}g\) for \(g\in\LHom_\dG(v,a)\) and 
      \(h\in\dX\) is the composition in \(\dG\) of the three elements 
      \(g,h,i\in\dG_2\), where \(i\) is the identity morphism of the object 
      \(f\) of \(\RHom_\dG(a,b)\). Similarly for the compositions 
      \(g'\circ{}h\), \(h'\circ{}g'\), \(g\circ{}h'\), \(h\circ{}h'\) and 
      \(h'\circ{}h\), for \(g'\in\LHom_\dG(v,b)\) and \(h'\in\dX^{-1}\) (in 
      each case, the two elements of \(\dG_2\) along with \(i\) form three 
      faces of a \(2\)-horn, with vertices \(a,a,b,v\) or \(a,b,b,v\), and 
      the result is the uniquely determined fourth face)
  \end{enumerate}
  It is clear, by construction, that each of the inclusions of \(\LHom(v,a)\) 
  and of \(\LHom(v,b)\) into \(\dH\) determine fully faithful functors. As in 
  the general proof, they are also essentially surjective by the Kan 
  property.
\end{example}

\begin{cor}\label{cor:comp}
  Let \(\dG\) be a \(2\)-groupoid, and let \(\gamma:D^2\ra\dG\) be a fixed 
  map, and \(a\in\dG_0\) a fixed vertex. Then
  \[\LHom(a,\gamma)=\LHom(a,\gamma\circ\Omit2)\cop_{\LHom(a,\gamma\circ{}1)}\LHom(a,\gamma\circ\Omit0)\]
  (canonical isomorphism), and
  \[\LHom(a,\gamma\circ\Omit1)=\LHom(a,\gamma\circ\Omit0)\circ_{\LHom(a,\gamma\circ{}1)}\LHom(a,\gamma\circ\Omit{2})\]
\end{cor}
In other words, the composition of two morphism groupoids (in the sense 
of~\S\ref{sss:composition}) is given by composition in the homotopy category.

\begin{proof}
  By definition, both sides have the same sets of objects.  
  Proposition~\ref{prp:1pushouts} provides the required map, and since on both 
  sides we also have a weak equivalence (by the second part of 
  Proposition~\ref{prp:1pushouts} and by Proposition~\ref{prp:hom}), this map is an 
  isomorphism. The second part again follows directly from the definition, as 
  both sides are the restriction to the same set of objects.
\end{proof}

\subsection{The theory associated with a groupoid}
We continue to fix \(n\in\NN\). To each definable \(n\)-groupoid in the 
theory \(\tT\) we define an associated expansion \(\tT_\dG\) of \(\tT\), 
directly generalising (a variant of) the one-dimensional case 
(Construction~\ref{con:cover}).

\begin{defn}\label{def:tg2}
  Let \(\dG\) be a definable \(n\)-groupoid in a theory \(\tT\). The 
  expansion \(\tT_\dG\) of \(\tT\) is obtained by adding additional sorts 
  \(\dG^*_i\) for \(0\le{}i\le{}n+1\), function symbols 
  \(e_i:\dG_i\ra\dG^*_i\), and a constant symbol \(*\in\dG^*_0\), and the 
  axioms expressing:
  \begin{enumerate}
    \item \(\dG^*\) is an \(n\)-groupoid, and \(e_*\) is a map of simplicial 
      sets (i.e., commutes with the structure maps). We identify \(\dG\) with 
      its image.
    \item \(\dG^*_0=\dG_0\cup\Set{*}\).
    \item The inclusion of \(\dG\) in \(\dG^*\) is a weak homotopy 
      equivalence (Remark~\ref{rmk:whe}), and an isomorphism onto the full 
      sub-groupoid of \(\dG^*\) spanned by \(\dG_0\).
  \end{enumerate}
  For each natural number \(r\), there is a definable family 
  \(\Gamma_r=\LHom(*,f)\) of groupoids, parametrised by the definable set of 
  maps \(f:D^r\ra\dG^*\). This is our collection \(\Gamma\) of admissible 
  groupoids, in the sense of Definition~\ref{def:2stable}.
\end{defn}

We note that our choice of \(\Gamma\) does satisfy the assumption on 
composition, by Corollary~\ref{cor:comp}.

As in the one-dimensional case, each object \(a\in\dG_0\) determines an 
interpretation \(\w_a\), over \(\tT_a\), determined by the requirement: 
\(\w_a(*)=a\), \(\w_a({\dG^*}_i)=\dG_i\) for \(i=1,2\) and similarly for the 
face maps. We would like to show that the \(\w_a\) are objects in the 
\(2\)-groupoid associated with \(\tT_\dG\) over \(\tT\), namely:

\begin{prop}\label{prp:2int}
  For every object \(a\in\dG_0\), the interpretation \(\w_a:\tT_\dG\ra\tT_a\) 
  is \(2\)-stable. In particular, \(\pp{\tT_\dG,\Gamma}\) is a \(2\)-internal 
  cover of \(\tT\).
\end{prop}
\begin{proof}
  We need to show that over some parameter \(u\), each \(\Gamma\)-admissible 
  groupoid \(\dH\) is equivalent to \(\w_a(\dH)\), over some parameters from 
  \(\tT\).  Let \(u\) be any element of \({\LHom_{\dG^*}(*,a)}_0\) (it is 
  consistent that such a \(u\) exists: for any model \(\Mod{M}\) of \(\tT\) 
  such that \(a\in\dG_0(\Mod{M})\), \(\Mod{M}\circ\w_a\) is a model of 
  \(\tT_{\dG}\) for which this set is non-empty).

  Let \(f:D^r\ra\dG^*\) be a map, and assume first that for some 
  \(i\in{}[r]\), \(b=f(i)\in\dG_0\). By the second item of 
  Proposition~\ref{prp:hom}, \(\dH=\LHom_{\dG^*}(*,f)\) is equivalent (over 
  no additional parameters) to \(\LHom_{\dG^*}(*,b)\), so we may assume that 
  \(f=b\). We may also assume that \(b\) is in the same connected component 
  as \(a\), because otherwise \(\dH\) is empty.  According to 
  Corollary~\ref{cor:hom}, it follows that \(\dH\) is equivalent 
  \(\LHom_{\dG^*}(*,a)\). Again according to (a dual version of) 
  Corollary~\ref{cor:hom}, the fixed element \(u\) determines an equivalence from 
  \(\dH\) to \(\w_a(\dH)=\LHom_\dG(a,a)\).

  The remaining case is when \(f\) is the constant map \(*\), so that 
  \(\dH=\LHom_{\dG^*}(*,*)\), and \(\w_a(\dH)=\LHom_\dG(a,a)\). The same 
  argument as above shows that both are equivalent to \(\LHom_{\dG^*}(*,a)\) 
  over \(u\).
\end{proof}

We would like to prove that the association \(a\mt\w_a\) is the object part 
of an assignment that recovers (up to equivalence) \(\dG\). To do that, we 
need to define the \(2\)-groupoid which is the target of this assignment.  
This will be the analog of \(I_{\tT*/\tT}\) from the one-dimensional case 
(\S\ref{sss:1igrpd}).

\begin{defn}\label{def:2catintern}
  Let \(\tT^*\) be a stable expansion of a theory \(\tT\), with admissible 
  family of distinguished internal covers \(\Gamma\).  The \(2\)-groupoid 
  associated to this datum is defined as follows:
  \begin{enumerate}
    \item Objects are \(2\)-stable interpretations of \(\pp{\tT^*,\Gamma}\) 
      in \(\tT\), over \(\tT\).
    \item If \(x,y\) are two objects as above, a morphism \(u:x\ra{}y\) is 
      given by a bi-interpretation \(u_{\tT'}:x(\tT')\ra{}y(\tT')\) over 
      \(\tT\), for each admissible internal cover \(\tT'\in\Gamma\). These 
      bi-interpretations are given with isomorphisms 
      \(c_i:u_{\tT_2}\circ{}x(i)\ra{}y(i)\circ{}u_{\tT_1}\) for every 
      admissible interpretation \(i:\tT_1\ra\tT_2\) between admissible covers 
      \(\tT_1,\tT_2\) in \(\Gamma\) (uniformly in families).

      The isomorphisms are required to satisfy: 
      \(c_{j\circ{}i}=y(j)(c_i)\circ{}c_j(x(i))\) for admissible 
      interpretations \(i:\tT_1\ra\tT_2\), \(j:\tT_2\ra\tT_3\) as above 
      (these make sense since, by definition, the \(c_i\) are definable maps 
      in \(y(\tT_2)\), \(y(j)\) is an interpretation of \(y(\tT_2)\) in 
      \(y(\tT_3)\) and \(c_j\) is a map between interpretations of 
      \(x(\tT_2\)), so can be applied to definable sets of the form 
      \(x(i)\).)
    \item The \(2\)-morphisms with edges \(u:x\ra{}y\), \(v:y\ra{}z\) and 
      \(w:x\ra{}z\) are given by isomorphisms \(v\circ{}u\ra{}w\), all over 
      \(\tT\).
    \item The ``\(2\)-composition'' of the \(2\)-morphisms 
      \(\alpha:v\circ{}u\ra{}w\), \(\beta:s\circ{}v\ra{}r\) and 
      \(\gamma:r\circ{}u\ra{}t\) is given by
      \[
      s\circ w\ra[s\cdot\alpha^{-1}]s\circ(v\circ u)=(s\circ v)\circ 
      u\ra[\beta\cdot u]r\circ u\ra[\gamma] t,
      \]
      where \(\cdot\) stands for pointwise application (or \emph{horizontal 
      composition}) as above.
  \end{enumerate}
  Applying the definition with \(\tT\) replaced by \(\tT_A\), for a 
  \(\tT\)-structure \(A\), we obtain a \(2\)-groupoid for each such structure 
  \(A\), which we denote \(I^2(A)={I^2}_{\tT^*/\tT}(A)\).
\end{defn}

Using the equivalence between internal covers and definable groupoids, this 
can be described in terms of definable groupoids. We give an explicit 
description in~\ref{sss:2gpd} below.

\begin{remark}\label{rmk:duskin}
  Let \(\cC\) be the category of definable groupoids in \(\tT\), with weak 
  equivalences as morphisms.  We may form its bi-category of \emph{co-spans} 
  for this category, as in~\Kero{0084}.  By Proposition~\ref{prp:bimor}, it 
  is equivalent (as a bi-category) to the category of internal covers and 
  bi-interpretations. The \(2\)-category in Definition~\ref{def:2catintern} 
  can be viewed as the \emph{Duskin nerve} (\Kero{009T}) of this bi-category 
  (this is clear from the description in~\Kero{00A1}). In particular, it 
  follows that this is indeed a \(2\)-category.
\end{remark}

\begin{prop}\label{prp:2gpdmap}
  Let \(\dG\) be a definable \(2\)-groupoid in a theory \(\tT\), and let 
  \(\tT_\dG=\pp{\tT_\dG,\Gamma}\) be the corresponding \(2\)-internal cover.  
  The association \(a\mt\w_a\) extends to a map 
  \(\w:\dG(A)\ra{I^2}_{\tT_\dG/\tT}(A)\) of \(2\)-groupoids, compatible with 
  extensions of the structure \(A\).
\end{prop}
\begin{proof}
  Proposition~\ref{prp:2int} shows that for all \(a\in\dG_0(A)\), \(\w_a\) is  
  indeed an object of \(I^2(A)\). Given \(t:a\ra{}b\) in \(\dG_1(A)\), we 
  define \(\w_t:\w_a\ra\w_b\) as follows:
  
  Let \(\dH\) be an admissible groupoid. As in the proof of~\ref{prp:2int}, 
  we may assume that \(\dH=\dH_c=\LHom_{\dG^*}(*,c)\) for some 
  \(c\in\dG_0(A)\), so that \(\w_a(\dH)=\LHom_{\dG}(a,c)\) and 
  \(\w_b(\dH)=\LHom_{\dG}(b,c)\). By Corollary~\ref{cor:hom}, \(t\) induces 
  an (admissible) equivalence from \(\w_a(\dH)\) to \(\w_b(\dH)\), which we 
  take to be \(\w_t(\dH)\). Our definition (and construction) ensures the 
  compatibility under admissible maps \(\dH\ra\dH'\).

  Similarly, let \(\alpha\in\dG_2(A)\), with edges \(r:a\ra{}b\), 
  \(s:b\ra{}c\) and \(t:a\ra{}c\). We need to construct an isomorphism (over 
  \(\tT\)) from \(\w_s\circ\w_r\) to \(\w_t\). Consider the map 
  \(f:D^2\ra\dG\) determined by \(\alpha\). We have \(f(01)=r\), \(f(12)=s\) 
  and \(f(02)=t\), so that for each object \(d\in\dG_0(A)\), the equivalence 
  \(\w_r(\dH_d):\w_a(\dH_d)\ra\w_b(\dH_d)\) is given by 
  \(\LHom(f\circ{}h_{01},d)\), where \(h_{01}:[1]\ra{}[2]\) is the inclusion 
  (and similarly for \(s,t\)). Hence, \(\LHom(f,d)\) represents the 
  composition \(\w_s\circ\w_r\), and restriction to \(\w_t\) provides the 
  required map.

  This completes the construction of \(\w\). The proof that this is a map of 
  \(2\)-groupoids (i.e., that it commutes with composition) is similar to the 
  above, using \(D^4\) in place of \(D^3\), and the fact that it commutes 
  with extension of scalars is obvious.
\end{proof}

\subsubsection{}\label{sss:2gpd}
Our main goal is to prove that the map \(\w\) constructed in 
Proposition~\ref{prp:2gpdmap} is a weak equivalence. Similarly to the 
\(1\)-dimensional case, it will generally only be true in a model. As a 
preparation, we consider more explicitly the structure of \(I^2\) from 
Definition~\ref{def:2catintern}, from a definable groupoid point of view.

Let \(\w_1\) and \(\w_2\) be two objects of \(I^2_{\tT^*/\tT}\), i.e., 
\(2\)-stable interpretations of \(\tT^*\) in \(\tT\). An isomorphism from 
\(\w_1\) to \(\w_2\) over \(\tT\) is given, according to 
Proposition~\ref{prp:spans}, by a family \(\dK(\dH)\) of groupoids in 
\(\tT\), for each admissible groupoid \(\dH\) in \(\tT^*\), along with weak 
equivalences \(u_i(\dH):\w_i(\dH)\weq\dK(\dH)\), all definable uniformly in 
\(\dH\). Given another admissible groupoid \(\dH'\), an admissible 
interpretation from \({\tT^*}_{\dH'}\) to \({\tT^*}_\dH\) is given, again by 
Proposition~\ref{prp:spans}, by an admissible groupoid \(\dX\) and admissible 
maps \(f:\dH\ra\dX\) and \(g:\dH'\weq\dX\).

According to Definition~\ref{def:2catintern}, we are provided with definable 
isomorphisms (realising the isomorphisms \(c_i\) there, via 
Proposition~\ref{prp:bimor})
\begin{equation}\label{eq:isostruct}
t_{\dX,\dK}:\dK(\dH)\circ_{\w_1(\dH)}\w_1(\dX)\iso
\w_2(\dX)\circ_{\w_2(\dH')}\dK(\dH')
\end{equation}
uniformly definable in \(\dX,\dK\) (and the associated embeddings), and 
restricting to the identity on \(\w_1(\dH')\) and on \(\w_2(\dH)\). The 
situation is depicted in the top part of the following diagram:
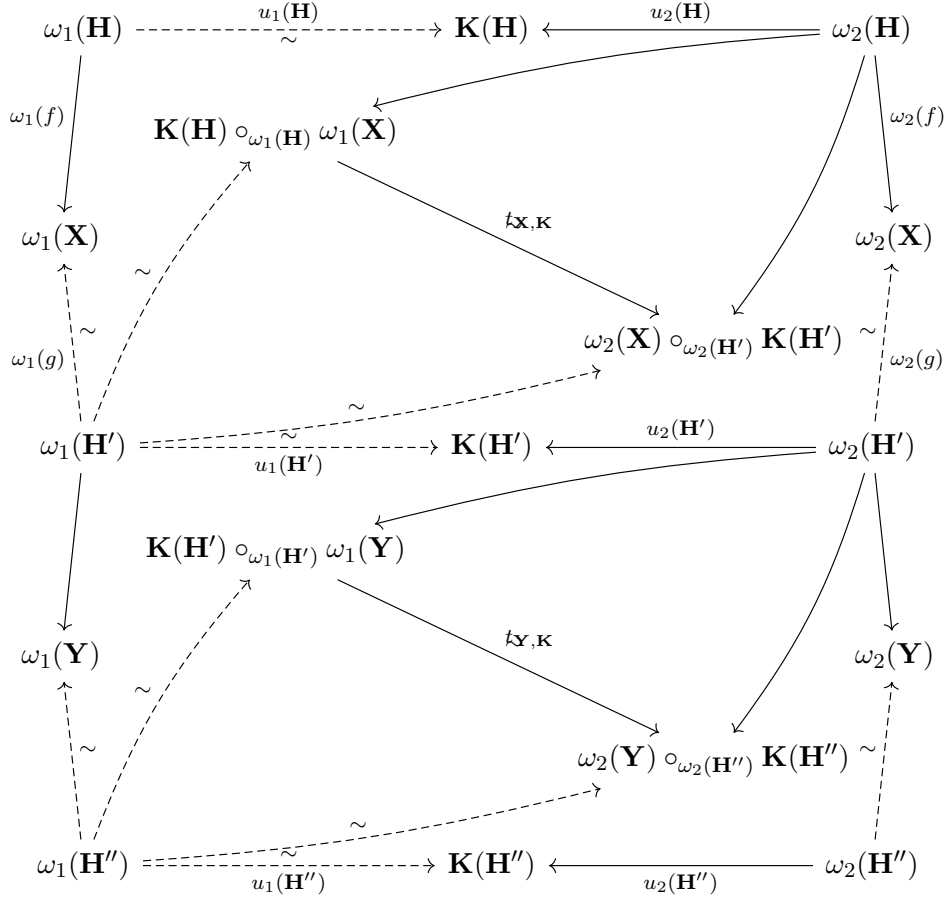
\begin{equation*}
  \begin{tikzcd}[column sep=tiny]
    &[-40pt] \w_1(\dH)\ar[rr, "u_1(\dH)", "\sim"', dashrightarrow]\ar[ldd, 
    "\w_1(f)"'] &[-10pt] & \dK(\dH) & &[-25pt]
    \w_2(\dH)\ar[ll, "u_2(\dH)"']\ar[rdd, "\w_2(f)"]
    \ar[llld, bend right=5]\ar[lddd, bend left=10]&[-40pt]\\
    & & \dK(\dH)\circ_{\w_1(\dH)}\w_1(\dX)\ar[rrdd,"t_{\dX,\dK}","\sim"] & & 
    & &\\
    \w_1(\dX) & & & & & & \w_2(\dX)\\
    & & & & \w_2(\dX)\circ_{\w_2(\dH')}\dK(\dH') &\\
    & \w_1(\dH')\ar[rr, "u_1(\dH')"', "\sim", dashrightarrow]
    \ar[ldd]
    \ar[luu, "\w_1(g)", "\sim"', dashrightarrow]
    \ar[ruuu,"\sim",dashrightarrow, bend left=10]
    \ar[rrru,"\sim",dashrightarrow, bend right=5]& & \dK(\dH') & &
    \w_2(\dH')
    \ar[rdd]
    \ar[llld, bend right=5]\ar[lddd, bend left=10]
    \ar[ll, "u_2(\dH')"']\ar[ruu, "\w_2(g)"', "\sim", dashrightarrow] &\\
    & & \dK(\dH')\circ_{\w_1(\dH')}\w_1(\dY)\ar[rrdd,"t_{\dY,\dK}","\sim"] & 
    & & &\\
    \w_1(\dY) & & & & & & \w_2(\dY)\\
    & & & & \w_2(\dY)\circ_{\w_2(\dH'')}\dK(\dH'') &\\
    & \w_1(\dH'')\ar[rr, "u_1(\dH'')"', "\sim", dashrightarrow]
    \ar[luu, "\sim"', dashrightarrow]
    \ar[ruuu,"\sim",dashrightarrow, bend left=10]
    \ar[rrru,"\sim",dashrightarrow, bend right=5]& & \dK(\dH'') & &
    \w_2(\dH'')\ar[ll, "u_2(\dH'')"]\ar[ruu, "\sim", dashrightarrow] &
  \end{tikzcd}
\end{equation*}

If \(\dY\) determines a map to \({\tT^*}_{\dH'}\) from \({\tT^*}_{\dH''}\) 
for a further groupoid \(\dH''\), we have the maps
\begin{multline}
  t_{\dX,\dK}\cop_{\w_1(\dH')}\1_{\w_1(\dY)}:
  \dK(\dH)\circ_{\w_1(\dH)}\w_1(\dX)\circ_{\w_1(\dH')}\w_1(\dY)\iso\\
  \w_2(\dX)\circ_{\w_2(\dH')}\dK(\dH')\circ_{\w_1(\dH')}\w_1(\dY)
\end{multline}
and
\begin{multline}
  \1_{\w_2(\dX)}\cop_{\w_2(\dH')}t_{\dY,\dK}:
  \w_2(\dX)\circ_{\w_2(\dH')}\dK(\dH')\circ_{\w_1(\dH')}\w_1(\dY)\iso\\
  \w_2(\dX)\circ_{\w_2(\dH')}\w_2(\dY)\circ_{\w_2(\dH'')}\dK(\dH'')
\end{multline}
the groupoid \(\dX\circ_{\dH'}\dY\) represents the composition of 
interpretations, and 
\(\w_i(\dX\circ_{\dH'}\dY)=\w_i(\dX)\circ_{\w_i(\dH')}\w_i(\dY)\) (canonical 
identification), since \(\w_i\) is an interpretation. Under this 
identification, we require that
\begin{equation}\label{eq:compat}
(\1_{\w_2(\dX)}\cop_{\w_2(\dH')}t_{\dY,\dK})\circ
(t_{\dX,\dK}\cop_{\w_1(\dH')}\1_{\w_1(\dY)})=t_{(\dX\circ_{\dH'}\dY),\dK}.
\end{equation}

Finally, a \(2\)-morphism is determined by a natural isomorphism between two 
maps as above (one a composition, which we already understand), so it is 
enough to describe those. Let \(\w_1\), \(\w_2\) and \(\dK\) be as above, and 
let \(\dL\) represent another morphism. A natural isomorphism is then given 
by a uniform family of isomorphisms \(\alpha_\dH:\dK(\dH)\ra\dL(\dH)\) over 
\(\w_i(\dH)\), which intertwine the maps \(t_{\dX,\dK}\) and \(t_{\dX,\dL}\) 
whenever \(\dX\) represents an interpretation. The \(2\)-composition of three 
such suitable maps is described as in Definition~\ref{def:2catintern}, with 
composition replaced by pushouts as appropriate.

\begin{remark}\label{rmk:2gpdmap}
  By definition, internality means that there is a non-empty definable 
  \emph{set} (i.e., a \(0\)-groupoid) of isomorphisms between the internal 
  sorts and sorts of the base theory.  Similarly, the structure described 
  above includes the description of a \(\tT^*\)-definable \(1\)-groupoid 
  \(\dIso_\tT(\ti{\tT^*},\ti{\tT})\) of weak equivalences between admissible 
  covers \(\ti{\tT^*}\) of \(\tT^*\) and covers \(\ti{\tT}\) of \(\tT\) 
  (non-empty for some \(\ti{\tT}\) if \(\tT^*\) is \(2\)-internal). In terms 
  of groupoids, the families \(\dK\) as in~\S\ref{sss:2gpd} are the objects, 
  and the morphisms are the natural isomorphisms \(\alpha\).  Furthermore, 
  this groupoid itself is admissible.
\end{remark}

\begin{example}\label{ex:2gpdmap}
  Let \(\tT^*=\tT_\dG\) as in Proposition~\ref{prp:2gpdmap}, and let 
  \(\w_1=\w_a\) and \(\w_2=\w_b\) for some \(a,b\in\dG_0(A)\). Let 
  \(f:a\ra{}b\) be a morphism in \(\dG(A)\) (identified with the 
  corresponding map from \(D^1\)).  Given an admissible groupoid 
  \(\dH=\dH_d=\LHom(d,*)\), we let \(\dK_f(d)=\dK_f(\dH_d)=\LHom_\dG(d,f)\), 
  with the canonical maps from \(\w_a(\dH_d)=\LHom(d,a)\) and 
  \(\w_b(\dH_d)=\LHom(d,b)\) (these are weak equivalences by 
  Proposition~\ref{prp:hom}).
  
  To give \(\dK\) the structure of an isomorphism from \(\w_a\) to \(\w_b\), 
  we need to supply the isomorphisms~\eqref{eq:isostruct}. If 
  \(\dH'=\dH_c=\LHom(c,*)\) is another admissible groupoid in \(\tT_\dG\), an 
  admissible isomorphism from \(\dH_c\) to \(\dH_d\) is given by a groupoid 
  \(\dX_g=\LHom(g,*)\), with \(g:c\ra{}d\) in \(\dG\). Seeing as 
  \(\w_x(\dX_g)=\LHom(g,x)\) for all \(x\), such a structure consists of a 
  definable family of maps
  \begin{multline*}
    t_{g,f}:\LHom(d,f)\circ_{\LHom(d,a)}\LHom(g,a)\iso\\
    \LHom(g,b)\circ_{\LHom(c,b)}\LHom(c,f)
  \end{multline*}
  Recall that \(t\) is the identity on objects, so we only need to define it 
  on morphisms.  Let \(\pp{u,v}\) represent a morphism of 
  \(\LHom(d,f)\circ_{\LHom(d,a)}\LHom(g,a)\). Let \(h:c\ra{}a\) be the domain 
  of \(v\). By the Kan property, there is a morphism \(w\in\LHom(c,f)\) whose 
  domain is \(h\). Then \(u,v,w\) form an element of \(\Horn{3}{3}\), so 
  composition provides a fourth face \(y\in{\LHom(g,b)}_1\).  We let 
  \(t_{g,f}(u\cop{}v)=y\cop{}w\).  If \(w'\) is a different choice in place 
  of \(w\), then \(w'\circ{}w^{-1}\) is in \(\LHom(c,b)\), so the result 
  represents the same morphism of \(\LHom(g,b)\circ_{\LHom(c,b)}\LHom(c,f)\).  
  It is clear that \(t\) is well defined on the class \(u\cop{}v\), and 
  uniformly definable in \(g,f\).

  To prove the identity~\eqref{eq:compat}, assume we are given another 
  morphism \(g':c'\ra{}c\), corresponding to an admissible interpretation 
  represented by \(\dY=\LHom(g',*)\). Let \(v'\in\LHom(g',a)=\w_1(\dY)\).  
  Proceeding with the notation above, we need to determine the image of 
  \(y\cop{}w\cop{}v'\) in 
  \(\LHom(g,b)\circ_{\LHom(c,b)}\LHom(g',b)\circ_{\LHom(c',b)}\LHom(c',f)\).  
  As above, it is given by \(y\cop{}y'\cop{}w'\), where 
  \(y'\in\LHom(g',b)=\w_2(\dY)\) and \(w'\in\LHom(c',f)=\dK_f(\dH_{c'})\) 
  represent the other two faces of a partial simplex with faces \(w\) and 
  \(v'\) (this other simplex can be visualised as attached to the previous 
  one at the face \(w\)). On the other hand, \(\dX\circ_{\dH_c}\dY\) was 
  identified (as in Corollary~\ref{cor:comp}) with \(\LHom(g\circ{}g',*)\), for 
  any composition \(g\circ{}g'\). After choosing such a composition \(h\), 
  \(v\cop{}v'\) is identified with an element of \(\LHom(h,a)\) and 
  \(y\cop{}y'\) with an element of \(\LHom(h,b)\), so that they become two 
  faces of the simplex with vertices \(a,b,c,d\), the other two being \(u\) 
  and \(w'\), so that \(u\cop{}w'=t_{h,f}(v\cop{}v',y\cop{}y')\), as 
  required.

  Assume now that we are given a map \(\gamma:D^2\ra\dG\) corresponding to an 
  element \(w\in\dG_2\), with edges \(f=\gamma(01)\), \(g=\gamma(12)\) and 
  \(h=\gamma(02)\).  Given element \(u\in\dK_f(\dH_c)=\LHom(c,f)\) and 
  \(v\in\dK_g(\dH_c)=\LHom(c,g)\) (for an arbitrary \(c\in\dG_0\)), the 
  \(2\)-composition applied to \(u,v\) and \(w\) provides an element of 
  \(\dK_h(\dH_c)=\LHom(c,h)\). This process assembles into a family of 
  isomorphisms 
  \(\alpha_\gamma:\dK_g(\dH_c)\circ\dK_f(\dH_c)\ra\dK_h(\dH_c)\), definable 
  uniformly in \(\gamma\) and \(c\). This completes the description (and a 
  reformulation of the proof) of the map in Proposition~\ref{prp:2gpdmap} in terms 
  of definable groupoids.
\end{example}

We are now ready to prove our main result.

\begin{theorem}\label{thm:main}
  Let \(\dG\) be a \(2\)-groupoid defined in a theory \(\tT\), and let 
  \(\tT_\dG\) be the associated theory (and admissible covers), as in 
  Definition~\ref{def:tg2}.  Then \(\tT_\dG\) is a \(2\)-internal cover of \(\tT\), 
  and for every model \(\mM\) of \(\tT\), the \(2\)-groupoid of 
  \(\mM\)-points \(I^2_{\tT_\dG/\tT}(M)\) is weakly equivalent to \(\dG(M)\).
\end{theorem}

\begin{proof}
  The fact that \(\tT_\dG\) is a \(2\)-internal cover is 
  Proposition~\ref{prp:2int}. The map from \(\dG\) to \(I^2_{\tT_\dG/\tT}\) was 
  constructed in Proposition~\ref{prp:2gpdmap}, and described in terms of groupoids 
  in Example~\ref{ex:2gpdmap}. We will use this description to show that the map 
  is a weak equivalence, using Remark~\ref{rmk:whe}. We assume that we are 
  working over \(\mM\), and proceed by considering the possible dimensions 
  \(0\le{}k\le{}2\).
  \begin{description}
    \item[\(k=0\)] We need to show that any \(2\)-stable interpretation 
      \(\w\) of \(\tT_\dG\) in \(\tT\) admits a coherent collection of 
      bi-interpretations as in Definition~\ref{def:2catintern}(2) 
      to some \(\w_a\).

      Since \(\w\) is an interpretation over \(\tT\), \(\w(\dG^*)\) is a 
      definable \(2\)-groupoid in \(\tT\), containing \(\dG\), with the 
      inclusion a weak equivalence. The proof now proceeds exactly as the 
      proof of Proposition~\ref{prp:2int}, with \(\w(\dG^*)\) in place of 
      \(\dG^*\).
    \item[\(k=1\)] This is the main case, which can be viewed as a definable 
      version of the Yoneda lemma. Let \(a,b\in\dG_0\), and assume we are 
      given an equivalence from \(\w_a\) to \(\w_b\).  Hence, for every 
      \(c\in\dG_0\) (some parameters), we are given a groupoid \(\dK(c)\) in 
      \(\tT\) and weak equivalences \(\LHom(c,a)=\w_a(\dH_c)\ra\dK(c)\) and 
      \(\LHom(c,b)=\w_b(\dH_c)\ra\dK_c\), uniformly in \(c\), along with 
      structure maps~\eqref{eq:isostruct}
      \[
      t_{g,\dK}:\dK(d)\circ_{\LHom(d,a)}\LHom(g,a)\ra
      \LHom(g,b)\circ_{\LHom(c,b)}\dK(c)
      \]
      (all notation as in Example~\ref{ex:2gpdmap}, except \(\dK\) is no longer 
      known to be of the given form). We identify \(\w_a(\dH_c),\w_b(\dH_c)\) 
      with their images in \(\dK(c)\).
      
      In particular, we have the identity morphism \(\1_a\) of \(a\) as an 
      object \(\1_a\in\w_a(\dH_a)\), and by weak equivalence, an object 
      \(f:a\ra{}b\) in \(\w_b(\dH_a)\subseteq{}\dK(a)\), along with a 
      morphism \(u:\1_a\ra{}f\) in \(\dK(a)\). We will show that \(\dK\) is 
      isomorphic to \(\dK_f\), by a unique isomorphism.

      To do that, let \(c\in\dG_0\) be an arbitrary object, and let \(v\) be 
      a morphism of \(\dK_f(c)=\LHom(c,f)\) (so a \(2\)-morphism of \(\dG\)).  
      Denote by \(g\in{\LHom(c,a)}_0\) the domain of \(v\). Then \(v\) can 
      also be viewed as a morphism in \(\LHom(g,b)\), and on the other hand, 
      we have the canonical morphism \(w\) from \(g\) to \(\1_a\) in 
      \(\LHom(g,a)\).  Applying \(t_{g,\dK}\) to the morphism 
      \(u\cop{}w\in\dK(a)\circ_{\LHom(a,a)}\LHom(g,a)\), we may write 
      \(t_{g,\dK}(u\cop{}w)\) as \(v\cop{}x\) for a unique \(x\in\dK(c)\), 
      which we take to be the image of \(v\). By construction this map 
      commutes with the structure maps \(t\), and is unique with this 
      property.
    \item[\(k=2\)] We need to show that each isomorphism 
      \(\alpha:\dK_g\circ\dK_f\ra\dK_h\) with \(f:a\ra{}b\), \(g:b\ra{}c\) 
      and  \(h:a\ra{}c\) arises from a unique \(\gamma:D^2\ra\dG\), with 
      boundary \(f,g,h\) (as in the end of Example~\ref{ex:2gpdmap}).  Uniqueness 
      was already shown in the part \(k=1\). For existence, we apply 
      \(\alpha_b:\dK_g(b)\circ\dK_f(b)\ra\dK_h(b)\) to the element 
      \(\1_g\cop\1_f\) (where \(\1_g\) is the identity morphism of the object 
      \(g\) of \(\LHom(b,c)\), viewed as an element of \(\dG_2\), and 
      similarly for \(f\)), to obtain an element \(\gamma_b\in{\dK_h(b)}_1\), 
      again viewed as a \(2\)-morphism of \(\dG\). It is clear that the map 
      \(\alpha\) coincides with \(\alpha_\gamma\) on the given maps, and 
      then, again by the uniqueness statement, that \(\alpha=\alpha_\gamma\) 
      globally.\qedhere
  \end{description}
\end{proof}

\subsubsection{Recovering a definable \(2\)-groupoid}
The main statement of classical internality starts with the assumption of 
internality, and produces a definable (non-empty) groupoid from it. The 
general outline of this construction was recalled in~\S\ref{sss:1grpd}, and 
in Proposition~\ref{prp:gsets} we indicated how this construction is useful 
in the description of definable sets in the cover.

In our approach, the construction of the (\(2\)-)groupoid is 
almost tautological: We defined a groupoid (or a \(2\)-groupoid) associated 
to every stable expansion, and by definition, the expansion is an internal 
cover if the groupoid is non-empty. However, we still need to show that the 
groupoid is equivalent to a definable one, which we sketch below. The other 
part, describing the (admissible) \(1\)-groupoids in the cover in terms of 
suitable definable fibrations in the base, is more involved, and we postpone 
most of the work here to future work.

\begin{prop}
  Let \(\tT^*\) be a \(2\)-internal cover of \(\tT\). Then the \(2\)-groupoid 
  \(I^2_{\tT^*/\tT}\) associated to it is equivalent to a \(\tT\)-definable 
  one.
\end{prop}
\begin{proof}[Proof sketch]
Assume \(\tT^*\) is a \(2\)-internal cover of \(\tT\), and let 
\(\w:\tT^*\ra\tT_A\) be a \(2\)-stable interpretation. For simplicity we 
assume that \(\Gamma\), the collection of admissible covers, consists of one 
definable family. As in~\S\ref{sss:2gpd}, we have a fixed parameter \(u_0\) 
and a uniform family \(\dK=\dK_c\) of groupoids in \(\tT^*\) defined over 
\(u_0\), along with (uniformly definable) weak equivalences 
\(f_c:\dX_c\weq\dK_c\) and \(g_c:\w(\dX_c)\ra\dK_c\) for \(\dX_c\) members of 
\(\Gamma\) (note that \(c\) ranges over a definable set in \(\tT\) by 
assumption). Like in Remark~\ref{rmk:2gpdmap}, as \(u_0\) varies, we obtain a 
family \(\dK_{u,c}\) of objects of a definable \(1\)-groupoid \(\dP_c\), 
along with a map \(\dP_c\ra\dIso(\dX,\w(\dX))\) for each member \(\dX\) of 
\(\Gamma\). Furthermore, \(\dP_c\) itself is also in \(\Gamma\). Applying the 
above map to \(\dX=\dP_{c'}\), we obtain a family of  definable maps of 
\(1\)-groupoids \(a:\dP_c\ra\dIso(\dP_{c'},\w_c(\dP_{c'}))\).

The \(2\)-groupoid \(\dG\) is constructed as follows: \(\dG_0\) is the 
definable set of parameters \(c\) as above. Each groupoid \(\dP_c\) will be 
isomorphic to \(\LHom(*,c)\) in the corresponding \(\dG^*\). Let \(c,d\) be 
two elements of \(\dG_0\). Given an object \(u\) of \(\dP_c\), the map \(a\) 
above produces a groupoid \(\dK_u\) as an object of 
\(\dIso(\dP_d,\w_c(\dP_d))\), along with weak equivalences 
\(f:\dP_d\weq\dK_u\) and \(g:\w_c(\dP_d)\ra\dK_u\). Let \(\dQ_{u,v}\) be the 
set of morphisms in \(\dK_u\) with domain \(f(v)\). We set the morphisms from 
\(c\) to \(d\) to be the definable types space of \(\dQ_{u,v}\) over \(\tT\) 
(this is definable in \(\tT\) by stability of the embedding). Note that each 
such type includes, in particular, the information of the object of 
\(\w_c(\dP_d)\) which is the codomain of any realisation (as an element of 
\(\dQ_{u,v}\)).

Similarly, assume \(e\) is another element of \(\dG_0\), and \(w\) an object 
of \(\dP_e\). The elements of \(\dG_2\) with vertices \(c,d,e\) are defined 
as the types over \(\tT\) of triples \(\dQ_{u,v}\x\dQ_{v,w}\x\dQ_{u,w}\) 
(over all such \(u,v,w\)). The \(2\)-composition is defined similarly, by 
considering \(4\)-tuples. We skip the details of the construction, as well as 
the proof that the map determined by \(a\) is a weak equivalence.
\end{proof}

We mention also that this description can, in principle, be used to give an 
equivalent combinatorial definition of \(2\)-internality (similar to the 
original definition of internality), but I could not find one sufficiently 
pleasant to write.

\subsection{Questions}
I mention a few natural questions that I hope to address in the future.

\subsubsection{Structure of admissible internal covers}
The definable version of the \(2\)-groupoid \(\dG\) associated to a 
\(2\)-internal cover \(\tT^*\) of \(\tT\) was only sketched above. Assuming 
it is properly described, it still needs to be seen that \(\tT_\dG\) and 
\(\tT^*\) are, in some sense, equivalent. It will also be useful to describe 
the admissible covers (in either) as suitably defined ``higher local 
systems'' on \(\dG\). Both questions require that we name the precise closure 
properties on the collection of admissible covers: we already assumed that 
they are closed under finite inverse limit and definable mapping spaces, but 
it is not clear, for example, if some closure under quantifiers is required.

\subsubsection{Lax interpretations}
We had not run into the questions above because we required to objects of the 
\(2\)-groupoid \(I^2\) to be actual interpretations. This works well in the 
example of \(\tT_\dG\), but for general expansions it might make more sense 
to consider a larger class of ``lax interpretations'' that preserve only the 
admissible covers (possibly up to weak equivalence).

\subsubsection{Internal covers of \(\tT_\dG\)}
The requirement for introducing admissible covers was motivated above.  
However, it might still be the case in the case of \(\tT_\dG\), essentially 
all internal covers are the ones described (up to covers that come from the 
base \(\tT\)). Again, stating this precisely requires clarifying the 
structure of the collection of admissible covers.

\subsubsection{Relation to analyzability}
The \(2\)-groupoid \(\dG^*\) in the theory \(\tT\) provides an example of a 
\(2\)-analyzable set over \(\tT\). Can we describe (combinatorially) which 
\(2\)-analyzable covers occur in this way?

\subsubsection{Generalization to higher dimensions}
This is rather clear: one continues by induction, defining an 
\(i+1\)-groupoid associated to a stable expansion by taking into account 
\(i\)-internal covers, and then defining the expansion to be an \(i+1\)-cover 
if this groupoid is non-empty. However, some of the proofs given above will 
be difficult to generalize, and it would interesting to look for a smoother 
way. In any case, this only applies to each finite level, and it does not 
seem reasonable to expect a generalization to arbitrary \(\infty\)-groupoids.

\subsubsection{Structure at \(*\)}
We did not consider the structure of the groupoid \(\LHom_{\dG^*}(*,*)\) 
definable in \(\tT_\dG\). On top of the groupoid structure, composition gives 
it a structure of a monoidal category up to homotopy (i.e., the homotopy 
category is monoidal). It also acts on all the admissible covers, so it is 
really a higher analog of the binding group. However, we did not consider 
what could be a version of the Galois correspondence or of descent, as in the 
\(1\)-dimensional case.

\printbibliography

\end{document}